\documentclass[10pt,a4paper,twoside,reqno]{amsart}

\usepackage[a4paper,
left=1in,right=1in,top=1in,bottom=1.2in,%
footskip=.25in]{geometry}

\usepackage{bm} 
\usepackage{latexsym, amsmath, amstext, amssymb, amsfonts, amscd, bm, array, multirow, amsbsy, mathrsfs}
\usepackage{amsthm}
\usepackage{t1enc}
\usepackage[mathscr]{eucal}
\usepackage{indentfirst}
\usepackage{pb-diagram}
\usepackage{graphicx}
\usepackage{fancyhdr}
\usepackage{fancybox}
\usepackage{color}
\usepackage{tikz-cd}
\usetikzlibrary{arrows}
\usepackage[all]{xy}
\usepackage{tikz}
\usepackage{xparse}
\usepackage{mathtools}
\usetikzlibrary{matrix}
\usepackage{upgreek}

\newcommand{\arxiv}[1]{{\tt
    \href{http://www.arXiv.org/abs/#1}{arXiv:#1}}}

\usepackage{url}
\usepackage[sort&compress,comma]{natbib}
\bibpunct{[}{]}{,}{n}{}{,}
\theoremstyle{plain}
\newtheorem{thm}{Theorem}[section]

\newtheorem{prop}[thm]{Proposition}
\usepackage{hyperref}
\hypersetup{colorlinks=false,pdfborderstyle={/S/U/W 0}}

\newtheorem{lemma}[thm]{Lemma}
\newtheorem{cor}[thm]{Corollary}

\theoremstyle{definition}
\newtheorem{definition}[thm]{Definition}

\theoremstyle{remark}
\newtheorem{remark}[thm]{Remark}

\newtheorem*{ack}{Acknowledgements}

\newcommand{\Alt}{\mathrm{Alt}}

\newcommand{\End}{\mathrm{End}}
\newcommand{\Aut}{\mathrm{Aut}}

\DeclareFontFamily{U}{rsf}{}
\DeclareFontShape{U}{rsf}{m}{n}{<5> <6> rsfs5 <7> <8> <9> rsfs7 <10-> rsfs10}{}
\DeclareMathAlphabet\Scr{U}{rsf}{m}{n}

\newcommand{\KA}{K\"{a}hler-Atiyah~}

\def\Z{\mathbb{Z}}

\def\R{\mathbb{R}}

\def\GL{\mathrm{GL}}
\def\dd{\mathrm{d}}

\def\fS{\mathfrak{S}}
\def\P{\mathbb{P}}
\def\even{\mathrm{even}}

\def\L{\mathrm{L}}
\def\Crit{\mathrm{Crit}}

\newcommand{\be}{\begin{equation*}}
\newcommand{\ee}{\end{equation*}}
\newcommand{\ben}{\begin{equation}}
\newcommand{\een}{\end{equation}}
\newcommand{\beqa}{\begin{eqnarray*}}
	\newcommand{\eeqa}{\end{eqnarray*}}
\newcommand{\beqan}{\begin{eqnarray}}
\newcommand{\eeqan}{\end{eqnarray}}

\newcommand{\nn}{\nonumber}

\newcommand{\threepartdef}[6]
{
	\left\{
	\begin{array}{ll}
		#1 & \mbox{~} #2 \\
		#3 & \mbox{~} #4 \\
		#5 & \mbox{~} #6
	\end{array}
	\right .
}

 \newcommand{\id}{\mathrm{id}}
 \newcommand{\tr}{\mathrm{tr}}
 
\newcommand{\sign}{\mathrm{sign}}

\def\cC{{\mathcal C}}

\def\cB{\Scr B}

\def\Cl{\mathrm{Cl}}

\def\Spin{\mathrm{Spin}}

\def\Pin{\mathrm{Pin}}
\def\Spin{\mathrm{Spin}}

\def\SO{\mathrm{SO}}
\def\O{\mathrm{O}}

\def\cA{\mathcal{A}}
\def\cE{\mathcal{E}}

\def\cC{\mathcal{C}}

\def\G_2{\mathrm{G_2}}

\def\cO{\mathcal{O}}

\def \cS{\mathcal{S}}

\def\P{\mathbb{P}}

\def\Aut{\mathrm{Aut}}

\def\Im{\mathrm{Im}}

\def\G{\mathrm{G}}
\def\L{\mathrm{L}}

\def\R{\mathbb{R}}

\def\rS{\mathrm{S}}

\def\cW{\mathcal{W}}
\def\cL{\mathcal{L}}

\def\dd{\mathrm{d}}

\def\Gl{\mathrm{Gl}}

\def\Spin{\mathrm{Spin}}

\def\Met{\mathrm{Met}}

\usepackage{enumerate}

\newcolumntype{P}[1]{>{\centering\arraybackslash}p{#1}}

\usepackage[math]{anttor}
\usepackage[T1]{fontenc}

\begin{document}

\title[A functional for Spin(7) forms]{A functional for \texorpdfstring{Spin(7)}{Spin(7)} forms}

\author[Calin Lazaroiu]{Calin Lazaroiu} \address{Departamento de Matem\'aticas, Universidad UNED - Madrid, Reino de Espa\~na}
\email{clazaroiu@mat.uned.es, lcalin@theory.nipne.ro} 

\author[C. S. Shahbazi]{C. S. Shahbazi} \address{Departamento de Matem\'aticas, Universidad UNED - Madrid, Reino de Espa\~na}
\email{cshahbazi@mat.uned.es} 
\address{Fakult\"at f\"ur Mathematik, Universit\"at Hamburg, Bundesrepublik Deutschland.}
\email{carlos.shahbazi@uni-hamburg.de}
 
\begin{abstract}
We characterize the set of all conformal $\Spin(7)$ forms on an oriented and spin Riemannian eight-manifold $(M,g)$ as solutions to a homogeneous algebraic equation of degree two for the self-dual four-forms of $(M,g)$. When $M$ is compact, we use this result to construct a functional whose self-dual critical set is precisely the set of all $\Spin(7)$ structures on $M$. Furthermore, the natural coupling of this potential to the Einstein-Hilbert action gives a functional whose self-dual critical points are conformally Ricci-flat Spin(7) structures. Our proof relies on the computation of the square of an irreducible and chiral real spinor as a section of a bundle of real algebraic varieties sitting inside the \KA bundle of $(M,g)$.  
\end{abstract}
 
\maketitle

\setcounter{tocdepth}{1} 
\tableofcontents


\section{Introduction}


Ever since their appearance in Berger's list \cite{Berger} of irreducible and non-symmetric Riemannian holonomy groups, $\Spin(7)$ holonomy metrics \cite{Bonan,Bryant}, and more generally, $\Spin(7)$ topological structures \cite{Fernandez}, have played and continue to play an increasingly important role in Riemannian geometry \cite{Bazzoni,Foscolo,HarveyL,Ivanov,Joyce,Joyce2007,KarigiannisFlows,Merchan,SalamonWalpuski}. Although the simply-connected and 21-dimensional compact Lie group $\Spin(7)$ can be nicely defined as the universal cover of $\SO(7)$ and can be realized explicitly as a subgroup of the group of units of the Clifford algebra in eight Euclidean dimensions \cite{Varadarajan}, the algebraic analysis of $\Spin(7)$ subgroups of the general linear group $\Gl(8,\mathbb{R})$ is notoriously difficult \cite{KarigiannisDefs,KarigiannisFlows}. This note attempts to shed some light on the algebra of $\Spin(7)$ forms through the study of the \emph{square} \cite{CLS} of an irreducible and chiral real spinor in eight Euclidean dimensions, whose stabilizer under the natural action of $\Spin(8)$ is necessarily a $\Spin(7)$ subgroup of the latter. Using this framework, in Theorem \ref{thm:Spin7algebraic} we obtain an algebraic characterization of conformal $\Spin(7)$ forms as solutions of a homogeneous algebraic equation for self-dual four-forms on an eight-dimensional oriented Euclidean vector space $(V,h)$. Solutions to this equation are precisely those $\Spin(7)$ structures on $V$ which induce metrics homothetic to $h$. Elaborating on this result, we show in Theorem \ref{thm:PotentialVh} that conformal $\Spin(7)$ forms on $(V,h)$ can be equivalently characterized as the critical points of a homogeneous algebraic function defined on the space of self-dual four-forms, which is equal to zero on such forms. The same function can be interpreted as a function on pairs $(h,\Phi)$ consisting of Euclidean metrics $h$ and four-forms $\Phi$. As proven in Theorem \ref{thm:MetricPotentialVh}, doing so determines a function whose self-dual critical set consists of all Spin(7) structures on the underlying vector space. Given a compact and spin eight-manifold $M$,these algebraic functions globalize to integrated potentials: the first one, denoted by $\cW_g$ for a Riemannian metric $g$, is homogeneous of degree three and is defined on the Fr\'{e}chet space of smooth self-dual four-forms of on $(M,g)$, whereas the second one, denoted by $\cW$, is defined on the Fr\'{e}chet space of smooth Riemannian metrics and four-forms. The critical points of $\cW_g$, which belong to its zero set, are in one-to-one correspondence with those $\Spin(7)$ forms on $M$ which induce metrics conformal to $g$. On the other hand, the self-dual critical points of $\cW$ consist of all Spin(7) structures on $M$. In propostion \ref{prop:EinsteinSpin7} we consider the natural coupling the Einstein-Hilbert action to $\cW$, and we prove that its self-dual critical set consists of pairs $(g,\Phi)$ given by a Ricci flat metric $g$ and a Spin(7) metric $\Phi$ that induces a metric conformally related to $g$. We note that four-forms in eight dimensions are not \emph{stable} in the sense of \cite{SatoKimura}, and therefore Hitchin's potential for three-forms in six and seven dimensions \cite{HitchinI,HitchinII} does not immediately extend to eight dimensions. Our results provide an intrinsic variational characterization of conformal $\Spin(7)$ structures, to the best of our knowledge the first of its kind in the literature, and does not require using any privileged basis or local model. We hope this potential could potentially be useful for the variational understanding of Spin(7) structures of a special type, ideally in the context of evolution flows for $\Spin(7)$ structures \cite{KarigiannisFlows,DwivediHarm,DwivediGrad,Fadel,Lobeau} and Donaldson-Thomas theory for $\Spin(7)$ manifolds \cite{DonaldsonSegal,DonaldsonThomas}.

\begin{ack} 
CSS would like to thank D. Joyce for useful pointers on a preliminary version of the draft. CIL is supported by a Maria Zambrano award and by grant PN 23-210101/2023. He also acknowledges support by COST action CA22113 ("Fundamental Challenges in Theoretical Physics"). The work of CSS was supported by the Leonardo grant LEO22-2-2155 of the BBVA Foundation. 
\end{ack}


\section{Preliminaries on Spin(7) forms}


This section recalls the basic algebraic theory of $\Spin(7)$ structures and $\Spin(7)$ forms. Let $V$ be an oriented eight-dimensional real vector space and let $V^{\ast}$ be its dual, which we endow with the orientation induced from $V$. Denote by $\Gl(V)=\Aut(V)$ the Lie group of linear automorphisms of $V$ and by $\Gl_+(V)$ its identity component, which can be realized as the group of orientation-preserving linear transformations of $V$.  


\subsection{$\Spin(7)$ structures on an eight-dimensional vector space}


By definition, a {\em $\Spin(7)$ structure} on $V$ is a $\Spin(7)$ subgroup of the group $\GL_+(V)$. A well-known way to describe such a structure is to give a non-zero four-form on $V$, called the {\em Cayley form} of the structure, with certain properties. We start by recalling this description. Consider first the real vector space $\mathbb{R}^8$ and denote by $(e_1,\hdots,e_8)$ its standard basis, with dual basis $(e^1,\hdots,e^8)$. Let $h_0$ be the usual Euclidean metric on $\mathbb{R}^8$ and endow this space with its canonical orientation, for which the Euclidean volume form reads:
\be
\nu_{o} = e^1\wedge e^2 \wedge e^3 \wedge e^4 \wedge e^5 \wedge e^6 \wedge e^7 \wedge e^8\, .
\ee
The standard $\Spin(7)_+$ structure on $\R^8$ is described by the {\em canonical Cayley form} $\Phi_0 \in \wedge^4(\mathbb{R}^8)^{\ast}$, which is defined as follows \cite{Bonan,HarveyL,Joyce2007,SalamonWalpuski}:
\begin{eqnarray}
& \Phi_0 := e^1\wedge e^2\wedge e^3 \wedge e^4+e^1\wedge e^2\wedge e^5 \wedge e^6 + e^1\wedge e^2\wedge e^7 \wedge e^8 +e^1\wedge e^3\wedge e^5 \wedge e^7 \nn\\
& - e^1\wedge e^3\wedge e^6 \wedge e^8 - e^1\wedge e^4\wedge e^5 \wedge e^8 - e^1\wedge e^4\wedge e^6 \wedge e^7 + e^5\wedge e^6\wedge e^7 \wedge e^8 + e^3\wedge e^4\wedge e^7 \wedge e^8 \label{Phi_0}\\
& + e^3\wedge e^4\wedge e^5 \wedge e^6 + e^2\wedge e^4\wedge e^6 \wedge e^8 - e^2\wedge e^4\wedge e^5 \wedge e^7 - e^2\wedge e^3\wedge e^5 \wedge e^8 - e^2\wedge e^3\wedge e^6 \wedge e^7\nn
\end{eqnarray}
This four-form is self-dual with respect to the Euclidean metric $h_0$ and the canonical orientation of $\R^8$. Furthermore, we have $\Phi_0\wedge \Phi_0 = 14\nu_0$ and hence the square norm of $\Phi_0$ with respect to $h_0$ is $\vert\Phi_0\vert_{h_0}^2 = 14$. The general linear group $\Gl(8,\mathbb{R})$ acts naturally on $\wedge^4(\mathbb{R}^8)^{\ast}$ and consequently on $\Phi_0$. The stabilizer of $\Phi_0$ under this action is isomorphic with the Lie group $\Spin(7)$ \cite{Bonan,HarveyL, Joyce2007,SalamonWalpuski}. The stabilizer preserves $h_0$ and $\nu_0$ and hence is a subgroup of the special orthogonal group $\SO(8)\subset \GL(8,\mathbb{R})$ determined by $h_0$ and $\nu_0$. 

\begin{definition}
\label{def:Spin7forms} 
A {\em $\Spin(7)_+$ form} on $V$ is a four-form $\Phi \in \wedge^4 V^{\ast}$ for which there exists an orientation-preserving linear isomorphism $f\colon V\to \mathbb{R}^8$ satisfying $\Phi = f^{\ast}\Phi_0$, where $f^{\ast}\colon \wedge(\mathbb{R}^8)^{\ast} \to \wedge V^{\ast}$ denotes the pull-back of forms by $f$. A {\em $\Spin(7)_{-}$ form} is defined similarly but using an orientation-reversing linear isomorphism $f$. A \emph{$\Spin(7)$ form} on $V$ is either a $\Spin(7)_+$  or a $\Spin(7)_-$ form defined on $V$.
\end{definition}

\noindent 
In particular, pulling back the canonical Cayley form by any orientation-reversing linear automorphism of $\R^8$ (for example, the reflection in the hyperplane with equation $x^8=0$) produces a $\Spin(7)_-$ form. A four-form $\Phi\in \wedge^4 V$ is a $\Spin(7)_+$ form on $V$ if and only if there exists a positively-oriented basis $(\epsilon_1,\ldots, \epsilon_8)$ of $V$ in which $\Phi$ is given by the relation obtained from \eqref{Phi_0} by replacing $e_i$ with $\epsilon_i= f^{-1}(e_i)$ for all $i=1 ,\ldots,8$.  Every $\Spin(7)_+$ form $\Phi$ comes together with an Euclidean metric on $V$ given by:
\be
h_\Phi = f^{\ast} h_0=\sum_{i=1}^8 \epsilon^i\otimes \epsilon^i
\ee
which makes $\epsilon_1,\ldots, \epsilon_8$ into an orthonormal basis. Similar statements hold for $\Spin(7)_-$ forms, except that the relevant bases of $V$ are negatively oriented. If $\Phi\in \wedge^4 V$ is a $\Spin(7)$ form, then $\lambda \Phi$ is also a $\Spin(7)$ form for any {\em positive} $\lambda\in \R_{>0}$.
 
A $\Spin(7)$ form $\Phi$ determines its associated metric $h_\Phi$ algebraically as explained in \cite[Section 4.3]{KarigiannisDefs}; we say that $h_\Phi$ is {\em induced by $\Phi$}. Let $\nu_{h_\Phi}=f^\ast \nu_0$ be the corresponding volume form. A $\Spin(7)_+$ form $\Phi$ is self-dual with respect to $h_\Phi$ in the given orientation of $V$ and satisfies $\Phi\wedge \Phi=14 \nu_{h_\Phi}$ (as can be verified for $\Phi_0$ on $\R^8$ and pulling back by $f$), a condition which amounts to $|\Phi|_{h_\Phi}=\sqrt{14}$. On the other hand, a $\Spin(7)_-$ form $\Phi$ is anti-self-dual and satisfies $\Phi\wedge \Phi=-14 \nu_{h_\Phi}$, which again amounts to $|\Phi|_{h_\Phi}=\sqrt{14}$. Notice that a $\Spin(7)$ form can induce {\em any} Euclidean metric on $V$. The stabilizer of a four-form $\Phi\in \wedge^4 V^\ast$ inside $\GL_+(V)$ is isomorphic with the group $\Spin(7)$ if and only if there exists a sign factor $\pm$ such that $\pm \Phi$ is a $\Spin(7)$ form. There are exactly two conjugacy classes of $\Spin(7)$ subgroups in $\Gl_+(V)$ \cite{Varadarajan}. The subgroups belonging to one of these stabilize $\Spin(7)_+$ forms, while those belonging to the other stabilize $\Spin(7)_-$ forms defined on $(V,h)$. Moreover, two $\Spin(7)$ forms have the same stabilizer if and only if they differ by multiplication with a {\em positive} real number. This establishes a bijection between the conjugacy class of $\Spin(7)$ subgroups of $\GL_+(V)$ and the set of positive homothety classes of $\Spin(7)$ forms. The conjugacy classes of $\Spin(7)_+$ and $\Spin(7)_-$ forms inside $\GL(V)$ combine into a single conjugacy class within $\Gl(V)$; for example, the reflection of $V$ in any hyperplane contained in $V$ conjugates $\Spin(7)_+$ forms to $\Spin(7)_-$ forms inside $\GL(V)$. 

\begin{remark}
If $\Phi$ is a $\Spin(7)_{+}$ form on $V$, then $\Phi$ and $-\Phi$ have the same stabilizer, which is a $\Spin(7)_+$ subgroup of $\GL_+(V)$. However, $-\Phi$ is {\em not} a $\Spin(7)_{+}$ form although it can be written as $f^\ast(-\Phi_0)$ for some orientation-preserving isomorphism $f:V\rightarrow \R^8$. A simple continuity argument shows that the overall sign of the canonical Cayley form cannot be changed by acting with an element of $\GL_+(8,\mathbb{R})$ on $\R^8$. It is traditional to fix the overall sign of the canonical Cayley form in order to avoid double counting of $\Spin(7)$ subgroups of $\GL_+(V)$.
\end{remark}


\subsection{$\Spin(7)$ structures on an eight-dimensional Euclidean space}


Let $(V,h)$ be an oriented Euclidean vector space of dimension eight. Let $\O(V,h)\subset \GL(V)$ be the disconnected group of all orthogonal transformations of $(V,h)$ and let $\SO(V,h)\subset \O(V,h)$ be its identity component. These groups act naturally on $\wedge V^\ast$. The Hodge operator $\ast_h$ of $h$ squares to the identity on four-forms and hence gives a decomposition:
\be
\wedge^4 V^{\ast} = \wedge^4_{+} V^{\ast} \oplus \wedge^4_{-} V^{\ast} 
\ee
where $\wedge^4_{+} V^{\ast}$ and $\wedge^4_{-} V^{\ast}$ are the eigenspaces of self-dual and anti-self-dual four forms, which correspond to the eigenvalues $+1$ and $-1$. We denote by $\langle \cdot ,\cdot \rangle_h$ the scalar product induced by $h$ on the exterior algebra: 
\be
\wedge V^\ast:=\oplus_{k=0}^8 \wedge^k V^\ast
\ee
and by $|\, \cdot \, |_h$ the corresponding norm. For later convenience, we will often work with the dual oriented Euclidean space $(V^\ast, h^\ast)$ instead of $(V,h)$. By definition, a {\em metric $\Spin(7)$ structure} on $(V,h)$ is a $\Spin(7)$ subgroup of $\GL_+(V)$ which is contained in $\SO(V,h)$. There exist two conjugacy classes of such subgroups in $\SO(V,h)$, which correspond to the two conjugacy classes in $\GL_+(V)$; they are called respectively the conjugacy classes of metric $\Spin(7)_+$ and metric $\Spin(7)_-$ structures on $(V,h)$. They combine into a single conjugacy class of $\O(V,h)$. 

\begin{definition}	
\label{def:Spin7formsII}
A {\em metric} $\Spin(7)$ form on $(V,h)$ is a $\Spin(7)$ form $\Phi$ on $V$ which satisfies $h_\Phi=h$. A \emph{conformal} $\Spin(7)$ form on $(V,h)$ is a $\Spin(7)$ form $\Phi$ on $V$ which satisfies $h_\Phi=\alpha\, h$ for some $\alpha>0$.
\end{definition}

\noindent 
The positive number $\alpha:=\alpha_\Phi$ is uniquely determined by the conformal $\Spin(7)$ structure $\Phi$; we call it the {\em conformal constant} of $\Phi$ relative to $h$. The stabilizer of a $\Spin(7)$ form $\Phi$ on $V$ is a metric $\Spin(7)$ structure on $(V,h)$ if and only if $\Phi$ is a conformal $\Spin(7)$ form on $(V,h)$. The conformal constant $\alpha_\Phi$ of a conformal  $\Spin(7)$ form on $(V,h)$ can be expressed through the norm $|\Phi|_h$ as follows. 

\begin{lemma}
\label{lemma:conformalconstant}
Let $\Phi\in \wedge^4 V^{\ast}$ be a conformal $\Spin(7)$ form on $(V,h)$. Then we have: 
\be
\alpha_\Phi=14^{-\frac{1}{4}} \vert\Phi\vert^{\frac{1}{2}}_h
\ee
Thus $\Phi$ is a metric $\Spin(7)$ form on the oriented Euclidean space $(V,h_\Phi)$, where: 
\ben
\label{hPhi}
h_\Phi=14^{-\frac{1}{4}} \vert\Phi\vert^{\frac{1}{2}}_h h
\een
is the metric induced by $\Phi$. In particular, a conformal $\Spin(7)$ form on $(V,h)$ is a metric $\Spin(7)$ form if and only if $\vert\Phi\vert_h = \sqrt{14}$.
\end{lemma}

\begin{proof}
Let  $\Phi\in \wedge^4 V^{\ast}$ be a conformal $\Spin(7)$ form on $(V,h)$ and let $f:V \rightarrow \R^8$ be a linear isomorphism such that $\Phi=f^\ast(\Phi_0)$. 
We have $h_\phi=f^{\ast}(h_0) = \alpha_\Phi h$, which implies:
\be
\sqrt{14} = |\Phi|_{h_\Phi}=|\Phi|_{\alpha_\Phi h} = (\alpha_\Phi)^{-2} |\Phi|_{h}  
\ee
This gives:
\be
\alpha_\Phi = 14^{-\frac{1}{4}} \vert\Phi\vert^{\frac{1}{2}}_h
\ee
Hence $h_\Phi$ equals \eqref{hPhi}. 
\end{proof}

\begin{lemma}
\label{lemma:PhiConfRay}
Let $\Phi\in \wedge^4 V^\ast$ be a four-form on $V$. Then $\Phi$ is a conformal $\Spin(7)$ form on $(V,h)$ if and only if there exists $\lambda\in \R_{>0}$ such that $\lambda \Phi$ 
\be
\frac{\sqrt{14}}{|\Phi|_h} \Phi\in \wedge^4_{+} V^{\ast}
\ee

\noindent
is a metric $\Spin(7)$ form on $(V,h)$. 
\end{lemma}

\begin{proof}
If $\Phi\in \wedge^4 V^{\ast}$ is a $\Spin(7)$ form on $V$, then so is $\lambda \Phi$ for every $\lambda \in \R_{>0}$. Indeed,  if $\Phi = f^{\ast} \Phi_0$ then $\lambda\Phi = (\lambda^{\frac{1}{4}} f)^{\ast} \Phi_0$. If $\Phi$ is a {\em conformal} $\Spin(7)$ form on $(V,h)$ then $h_\Phi=f^\ast(h_0)$ is proportional to $h$ and hence $h_{\lambda\Phi}=(\lambda^{\frac{1}{4}} f)^\ast(h_0)=\lambda^{1/2} h_\Phi$ is also proportional to $h$, which shows that $\lambda\Phi$ is a conformal $\Spin(7)$ form. The latter is a metric $\Spin(7)$ form on $(V,h)$ when $h_{\lambda\Phi}=h$, which happens if and only if:
\begin{equation*}
\lambda=\left(\frac{h}{h_\Phi}\right)^2=\frac{\sqrt{14}}{|\Phi|_h}
\end{equation*}

\noindent
where in the last equality we used Lemma \ref{lemma:conformalconstant}.
\end{proof}

\noindent 
Lemma \ref{lemma:PhiConfRay} implies: 

\begin{cor}
\label{cor:conformal}
The set of conformal $\Spin(7)_\pm$ forms on $(V,h)$ is in bijection with the set of positive homothety classes of metric $\Spin(7)_\pm$ structures on $(V,h)$. 
\end{cor}

\noindent
The stabilizers of both metric and conformal $\Spin(7)_\pm$ forms on $(V,h)$ are the metric $\Spin(7)_\pm$ structures on $(V,h)$. Each $\Spin(7)_\pm$ subgroup of $\SO(V,h)$ corresponds to the positive homothety class of a metric $\Spin(7)_\pm$ form, which consists of conformal  $\Spin(7)_\pm$ forms. In particular, two conformal $\Spin(7)_\pm$ forms on $(V,h)$ have the same stabilizer inside $\SO(V,h)$ if and only if they differ through multiplication by a constant. Two metric $\Spin(7)_\pm$ forms on $(V,h)$ have the same stabilizer inside $\SO(V,h)$ if and only if they coincide. 

We denote by $\fS_\pm(V,h)$ the conjugacy classes of metric $\Spin(7)_\pm$ structures, by $\cA_\pm(V,h)\subset \wedge^4_+ V^\ast$ the set of {\em metric} $\Spin(7)_\pm$ forms and by $\cC_\pm(V,h)\subset \wedge^4_+ V^\ast$ the set of {\em conformal} $\Spin(7)_\pm$ forms on $(V,h)$. The special orthogonal group $\SO(V,h)\simeq \SO(8)$ acts transitively on $\fS_\pm(V,h)$ (see \cite{Varadarajan}). Hence fixing a metric $\Spin(7)_\pm$ structure gives a bijection:
\ben
\label{cAmetric}
\fS_\pm(V,h) \simeq \SO(8)/\Spin(7) \simeq \mathbb{RP}^7
\een
Since metric $\Spin(7)_\pm$ structures are in bijection with metric $\Spin(7)_\pm$ forms, \eqref{cAmetric} gives a bijection: 
\be
\cA_\pm(V,h)\simeq \mathbb{RP}^7
\ee
Let: 
\ben
\label{cC}
C_\pm(V,h):=\{\lambda \Phi\,\, \vert\,\, \lambda\in \R_{\geq 0} \,\, \& \,\,\Phi\in \cA_\pm(V,h)\}\subset \wedge^4_+ V^\ast 
\een
be the real affine cone over $\cA_\pm(V,h)$, which is diffeomorphic with $\R^8/\Z_2$. Then Corollary \ref{cor:conformal} shows that the set $\cC_\pm(V,h)$ of conformal $\Spin(7)_\pm$ forms on $(V,h)$ identifies with the pointed cone obtained by removing the origin from $C_\pm(V,h)$.


\section{Spinor polyform squares on an eight-dimensional oriented Euclidean space}
\label{sec:spinpoly}


In this section, we use an algebraic characterization of the square of a real and irreducible chiral spinor in eight Euclidean dimensions, which follows from the general theory developed in \cite{CLS}, to obtain an intrinsic quadratic description of $\Spin(7)_\pm$ forms on an oriented Euclidean space $(V,h)$ of dimension eight. We use this result to construct a homogeneous algebraic potential defined on the space $\wedge^4_+ V^\ast$ which vanishes on $\Spin(7)_+$ forms. We begin with a brief introduction to the K\"{a}hler-Atiyah model of the Clifford algebra $\Cl(V^{\ast},h^{\ast})$. This formalism was used extensively in references \cite{ga1,ga2,gf,fol1,fol2}, where it was shown to play an important role in understanding certain geometric problems that arise in string and M theory. This inspired both \cite{CLS} as well as the potential that we construct below \cite{gaptalk}. 


\subsection{The K\"ahler-Atiyah algebra}


As mentioned earlier, it is useful to consider the dual quadratic space $(V^{\ast},h^{\ast})$ of the eight-dimensional oriented Euclidean space $(V,h)$, which we endow with the orientation induced from $V$. We denote by $\langle \cdot ,\cdot \rangle_h$ the scalar product induced by $h$ on the space $\wedge V^\ast$ of polyforms defined on $V$ and by $|\, \cdot \, |_h$ the corresponding norm. Notice that $|\nu_h|_h= 1$. Let $\Cl(V^{\ast},h^{\ast})$ be the real Clifford algebra of $(V^{\ast},h^{\ast})$, which we identify with the {\em K\"{a}hler-Atiyah algebra} $(\wedge V^{\ast}, \diamond_h)$ of $(V,h)$. The latter is an associative and unital $\R$-algebra obtained by equipping $\wedge V^{\ast}$ with the \emph{geometric product} $\diamond_h\colon \wedge V^{\ast}\times \wedge V^{\ast}\to \mathbb{R}$. The geometric product is an $h$-dependent deformation  of the wedge product obtained by the $\R$-linear and associative extension of the following expression:
\be
\alpha\diamond_h \beta = \alpha \wedge \beta + \iota_{\alpha}\beta~,
\ee
where $\alpha \in V^{\ast}$ is a one-form, $\beta \in \wedge V^{\ast}$ is any polyform and $\iota_{\alpha}\beta$ denotes the contraction of $\beta$ with the vector $\alpha^{\sharp_h}\in V$ which is $h$-dual to the one-form $\alpha$. The unit of the algebra $(\wedge V^{\ast}, \diamond_h)$ is the unit $1\in \R=\wedge^0 V^{\ast}\subset \wedge V^{\ast}$. This algebra has a natural $\Z_2$-grading given by the direct sum decomposition of $\wedge V^\ast$ into the subspaces:
\be
\wedge^{\mathrm{even}} V^\ast :=\oplus_{k=0}^4 \wedge^{2k} V^\ast~~,~~\wedge^{\mathrm{odd}} V^\ast:=\oplus_{k=0}^4 \wedge^{2k+1} V^\ast 
\ee
As an associative and unital algebra, the K\"ahler-Atiyah algebra is isomorphic to the Clifford algebra $\Cl(V^{\ast},h^{\ast})$ through the $h$-dependent {\em Chevalley-Riesz isomorphism} (see \cite{ga1,gf,CLS}), which we denote by: 
\ben
\label{eq:CR}
\Psi_h \colon \Cl(V^{\ast}, h^{\ast})\stackrel{\sim}{\to} (\wedge V^{\ast},\diamond_h) 
\een
We denote by $\pi$ the {\em signature automorphism} of the \KA algebra, which is defined as the unique unital algebra automorphism which acts as minus the identity on $V^{\ast}\subset \wedge V^{\ast}$ and by $\tau$ the {\em reversion anti-automorphism}, defined as the unique unital algebra anti-automorphism which acts as the identity on $V^{\ast}$. For $\alpha\in \wedge^k V^\ast$, we have:
\ben
\label{pitaurels}
\pi(\alpha)=(-1)^k \alpha \quad,\quad \tau(\alpha)=(-1)^{\frac{k(k-1)}{2}}\alpha 
\een

\noindent
The volume form $\nu_h$ of $h$ satisfies $\nu_h\diamond_h \nu_h=1$ as well as: 
\be
\nu_h \diamond_h \alpha=\pi(\alpha)\diamond_h \nu_h \quad \forall \alpha\in \wedge^\ast V^\ast\, . 
\ee
Moreover, we have \cite{ga1,gf,CLS}: 
\be
\ast_h\alpha=\tau(\alpha)\diamond_h \nu_h \quad \forall \alpha\in \wedge^\ast V^\ast\, .
\ee
In particular, a 4-form $\Phi\in \wedge^4 V^\ast$ is self-dual if and only if it satisfies: 
\be
\nu_h\diamond_h \Phi= \Phi\, .
\ee
The geometric product $\diamond_h$ admits an expansion as a sum of {\em generalized products}, which we describe next. 

\begin{definition}
The {\em contracted wedge products} are the $\R$-bilinear maps:
\be
\wedge_k: \wedge V^\ast \times \wedge V^\ast \rightarrow \wedge V^\ast\, , \quad \forall k\in \{0,\ldots, 8\}
\ee
defined by $\wedge_0=\wedge$ and the recursion relation: 
\be
\alpha\wedge_{k+1} \beta= h^{ab} \iota_{e_a}\alpha \wedge_k \iota_{e_b} \beta \quad \forall\alpha,\beta\in \wedge V^\ast \quad \forall  k\in \{0,\ldots, 7\}
\ee
where $(e_1,\ldots, e_8)$ is an arbitrary basis of $V$, $h^{ab}:=h^\ast(e^a,e^b)$ and we use implicit summation over repeated indices.  
\end{definition}

\noindent It is easy to check that $\wedge_k$ does not depend on the basis used in the definition. 

\begin{definition}
For any $k\in \{0,\ldots,8\}$, the $k$-th {\em generalized product} 
\be
\Delta^h_k\colon \wedge V^{\ast}\times \wedge V^{\ast} \rightarrow \wedge V^{\ast}
\ee
is the $\R$-bilinear map defined through:
\be
\Delta^h_k  = \frac{1}{k!} \wedge_k~ 
\ee
\end{definition}

\noindent For any polyforms $\alpha,\beta\in \wedge V^\ast$, we have $\alpha\Delta^h_0 \beta=\alpha\wedge \beta$.  For any basis $(e_1,\ldots, e_8)$ of $(V,h)$, we have: 
\ben
\label{DeltaExp}
\alpha\Delta_k^h\beta=\frac{1}{k!} h^{a_1b_1}\ldots h^{a_k b_k}(\iota_{e_{a_1}}\ldots \iota_{e_{a_k}} \alpha) \wedge (\iota_{e_{b_1}}\ldots \iota_{e_{b_k}}\beta) 
\een
Notice that $\alpha\Delta_k^h \beta\in \wedge^{p+q-2k} V^\ast$ if $\alpha\in \wedge^p V^\ast$ and $\beta\in \wedge^q V^\ast$. This vanishes if $k>\min(p,q)$ and in particular if $2k>p+q$. Moreover, we have (this follows by direct computation or from \cite[eq. (2.13)]{ga1}): 
\ben
\label{gradedasym}
\alpha\Delta_k^h \beta=(-1)^{pq+k} \beta\Delta_k^h \alpha\quad \forall \alpha\in \wedge^p V^\ast~~\forall \beta\in \wedge^q V^\ast 
\een
Hence $\Delta_k^h$ is graded-symmetric (with respect to the $\Z_2$-grading of the \KA algebra) when $k$ is even and graded-antisymmetric when $k$ is odd. An observation which will is relevant for this paper is that the four-form $\Phi\Delta_2^h\Phi$ is self-dual for any 
self-dual 4-form $\Phi\in \wedge^4_+ V^\ast$, as follows from \cite[eq. (A.3)]{fol1}: 
\ben
\label{PhiDeltaPhisd}
\Phi\in \wedge^4_+ V^\ast\Longrightarrow \Phi\Delta_2^h\Phi\in \wedge^4_+ V^\ast~.
\een
We refer the reader to \cite{ga1,gf,CLS} and to \cite[Appendix A]{fol1} for more detail and further properties of generalized products. The following result expresses the geometric product in terms of generalized products, a description that turns out to be very useful in applications. 

\begin{prop} \rm{\cite{ga1,gf,CLS}}
\label{prop:genprodexp}
Let $(e_1,\ldots, e_8)$ be any basis of $V$. For any pure rank forms $\alpha\in \wedge^p V^\ast$ and $\beta\in \wedge^q V^\ast$, we have: 
\ben
\label{DiamondExp}
\alpha\diamond_h\beta=\sum_{k=0}^d (-1)^{\left[\frac{k+1}{2}\right]+ k p} \alpha\Delta_k^h\beta
\een
where $\left[\, \cdot \, \right]$ denotes the integer part. This gives the expansion of the left hand side into pure rank forms. 
\end{prop}

\noindent
It is easy to check that the geometric product $\diamond_h:\wedge V^\ast\times \wedge V^\ast\rightarrow \wedge V^\ast$ is equivariant under the natural action of $\SO(V,h)$ on its domain and co-domain, a fact which also follows from the fact that the Kahler-Atiyah algebra is a model of the Clifford algebra $\Cl(V^\ast,h^\ast)$. The expansion \eqref{DiamondExp} implies that the generalized products are also equivariant under these actions. The following endomorphism of $\wedge^4 V^\ast$ will be used later on.

\begin{definition}
\label{def:Lambda}
For any $\Phi\in \wedge^4 V^\ast$, define $\Lambda^h_\Phi\in \End(\wedge^4 V^\ast)$ through:
\ben
\label{Lambda}
\Lambda^h_\Phi(\omega):=2 \Phi \Delta_2^h \omega\quad\forall \omega\in \wedge^4 V^\ast 
\een
\end{definition}

\begin{remark}
\label{rem:Lambda}
In an arbitrary basis $(e^1,\ldots, e^8)$ of $\wedge^4 V^\ast$, we have: 
\ben
\label{PhiDeltaomega}
\Phi\Delta_2^h \omega=\frac{1}{8}\Phi_{ijmn}\omega^{mn}_{\,\,\,\,\,\,\,\,\,\, kl} e^i\wedge e^j\wedge e^k\wedge e^l\quad\forall \Phi,\omega\in \wedge^4 V^\ast 
\een
where: 
\be
\omega^{m n}_{\,\,\,\,\,\,\,\,\, kl}=h^{mp} h^{nq} \omega_{pqkl}
\ee
and we expanded $\Phi$ and $\omega$ as: 
\be
\Phi=\frac{1}{4!}\Phi_{ijkl} e^i\wedge e^j\wedge e^k\wedge e^k~~,~~\omega=\frac{1}{4!}\omega_{ijkl} e^i\wedge e^j\wedge e^k\wedge e^k 
\ee
Here use the so-called "det convention" for the wedge product of forms. Hence the wedge product of $\alpha\in \wedge^p V^\ast$ and $\beta\in \wedge^q V^\ast$ is defined as in \cite[Chap. 7]{Spivak}:
\be
\alpha\wedge\beta=\frac{(p+q)!}{p!q!}\Alt(\alpha \otimes \beta)
\ee
where $\Alt$ is the projector onto the subspace of totally antisymmetric tensors. Using \eqref{PhiDeltaomega}, it is easy to check that the operator $\Lambda^h_\Phi$ defined in \eqref{Lambda} coincides with the operator denoted by $\Lambda_\Phi$ in \cite[Definition 2.7]{KarigiannisFlows} and \cite[Definition 2.5]{DwivediGrad}. 
\end{remark}


\subsection{The \KA trace}


Let us fix an irreducible representation $\gamma\colon \Cl(V^{\ast},h^{\ast}) \to \End(\Sigma)$ of the associative algebra $\Cl(V^{\ast},h^{\ast})$ on a real vector space $\Sigma$, which is necessarily of dimension $2^4=16$. This representation is unique up to equivalence; moreover, the map $\gamma$ is bijective (see \cite{gf}) and hence gives an isomorphism of unital and associative algebras from $\Cl(V^\ast, h^\ast)$ to $(\End(\Sigma),\circ)$, 
where $\circ$ denotes the composition of endomorphisms of $\Sigma$. Composing $\gamma$ with the inverse of the Chevalley-Riesz isomorphism \eqref{eq:CR} gives an isomorphism of unital associative algebras which we denote by:
\ben
\label{Psih}
\Psi_h^\gamma:=\gamma\circ \Psi_h^{-1}: (\wedge V^\ast,\diamond_h)\rightarrow  (\End(\Sigma),\circ) 
\een
This provides a particularly convenient model of the Clifford representation $\gamma$, which is useful from the point of view of spin geometry. We equip $\Sigma$ with a scalar product\footnote{Such a scalar product is determined up to multiplication by a positive constant. It belongs to one of the two homothety classes of admissible pairings (in the sense of \cite{ACD,AC}) which exist on an irreducible representation of $\Cl(V,h)$ in a real vector space.} $\cB:\Sigma\times \Sigma \rightarrow \mathbb{R}$ relative to which the operator of Clifford multiplication by co-vectors is symmetric (see \cite{ACD, AC, ga1,gf,CLS}). This scalar product satisfies:
\ben
\label{eq:cBadj}
\cB(\Psi_h^\gamma(\alpha)\xi_1 , \xi_2) = \cB(\xi_1 , \Psi_h^\gamma(\tau(\alpha))\xi_2)\quad \forall\,\, \xi_1, \xi_2 \in \Sigma \quad \forall \alpha\in \wedge V^{\ast}~,
\een
that is: 
\ben
\label{gammat}
\Psi_h^\gamma(\alpha)^t=\Psi_h^\gamma(\tau(\alpha))~~ \quad \forall \alpha\in \wedge V^\ast~,
\een
where $~^t:\End(\Sigma)\rightarrow \End(\Sigma)$ denotes the $\cB$-transpose of endomorphisms of $\Sigma$. Relation \eqref{eq:cBadj} implies that $\cB$ is invariant under the restriction of $\gamma$ to a representation of the pin group  $\Pin (V^{\ast}, h^{\ast}) \subset \Cl(V^{\ast},h^{\ast})$ on $\Sigma$. Using the isomorphism $\Psi_\gamma$, we endow the algebra $(\wedge V^\ast, \diamond_h)$ with the {\em K\"ahler-Atiyah trace} \cite{ga1,gf,CLS} $S:\wedge V^\ast\rightarrow \R$, which is defined by transporting the ordinary trace $\tr:\End(\Sigma)\rightarrow \R$ to $\wedge V^\ast$: 
\ben
\label{Str}
S(\alpha):=\tr(\Psi_h^{\gamma}(\alpha)) \quad \forall\,\, \alpha\in \wedge^\ast V^\ast 
\een
The K\"ahler-Atiyah trace satisfies $S(1)=\tr(\id_\Sigma)=16$ and the cyclicity property:
\ben
\label{cyclicity}
S(\alpha_1\diamond_h \alpha_2)=S(\alpha_2\diamond_h\alpha_1) \quad  \forall\,\, \alpha_1,\alpha_2\in V^\ast
\een
as well as the relation: 
\ben
\label{Stau}
\cS(\tau(\alpha))=\cS(\alpha)~~\forall \alpha\in \wedge V^\ast
\een
both of which follow from the similar properties of the map $\tr:\End(\Sigma)\rightarrow \R$. Moreover, $\cS$ is a non-degenerate trace on the \KA algebra since $\Psi_h^\gamma$ is bijective and $\tr$ is a non-degenerate trace on the algebra $(\End(\Sigma),\circ)$. Thus $(\wedge^\ast V^\ast, \diamond_h, S)$ is a non-commutative Frobenius algebra, which identifies with the noncommutative Frobenius algebra $(\End(\Sigma), \circ, \tr)$ through the isomorphism \eqref{Psih}. 

\begin{prop}
For any $\alpha\in \wedge V^\ast$, we have: 
\be
\cS(\alpha)=16 \alpha^{(0)}
\ee
where $\alpha^{(0)}$ denotes the rank zero component of the polyform $\alpha$.
\end{prop}

\begin{proof}
The $\R$-linear map $\cS$ is invariant under the natural action of $\SO(V,h)$ on $\wedge V^\ast$, because this action transports through $\Psi_h^\gamma$ to the conjugation action on operators, which preserves the ordinary trace on $\End(\Sigma)$. Since $\SO(V,h)$ acts non-trivially on forms of non-zero rank and since $\Psi_h(1)=\id_\Sigma$, this implies:
\be
\cS(\alpha)=\cS(\alpha^{(0)})=\alpha^{(0)} \cS(1)=16  \alpha^{(0)} \quad \forall\alpha \in \wedge V^\ast  
\ee
\end{proof}

\noindent
The \KA trace induces a new scalar product on $\wedge V^\ast$, which, as we shall see in a moment, corresponds to the scalar product induced by $\cB$ on $\End(\Sigma)$.

\begin{definition}
The {\em Frobenius pairing} is the scalar product  $Q$ defined on $\wedge V^\ast$ through: 
\ben
\label{Q}
Q(\alpha,\beta):=S(\tau(\alpha)\diamond_h \beta) = 16 (\tau(\alpha)\diamond_h \beta) ^{(0)} \quad \forall\,\, \alpha,\beta\in \wedge V^\ast~  
\een
\end{definition}

\noindent 
The fact that $Q$ is symmetric follows from the properties \eqref{cyclicity} and \eqref{Stau} of the \KA trace using the fact that $\tau$ is anti-automorphism. The nondegeneracy of $Q$ follows from that of $\cS$. 

\begin{prop}
The Frobenius pairing $Q$ corresponds through the isomorphism \eqref{Psih} to the scalar product $(~,~)_{\cB}$ induced on $\End(\Sigma)$ by the scalar product $\cB$:
\be
Q(\alpha,\beta)=(\Psi_h^\gamma(\alpha)^t, \Psi_h^\gamma(\beta))_{\cB}=\tr(\Psi_h^\gamma(\alpha)^t\circ \Psi_h^\gamma(\beta))\quad \forall \alpha,\beta\in \wedge V^\ast 
\ee
\end{prop}

\begin{proof}
The statement follows from the relation: 
\be
Q(\alpha,\beta)=S(\tau(\alpha) \diamond_h \beta)=\tr(\Psi^{\gamma}_h(\alpha)^t\circ\Psi^{\gamma}_h(\beta))=(\Psi^{\gamma}_h(\alpha), \Psi^{\gamma}_h(\beta))_{\cB}\quad \forall\,\, \alpha,\beta\in \wedge^\ast V^\ast 
\ee
where we used the definition: 
\be
(A,B)_{\cB}:=\tr(A^t \circ B)\quad \forall A,B\in \End(\Sigma)
\ee
of the scalar product induced by $\cB$ on $\End(\Sigma)$.
\end{proof}

\noindent 
Denote by $||.||_Q$ the norm induced by $Q$ on $\wedge V^\ast$. Sub-multiplicativity of the norm induced by $\cB$ immediately implies sub-multiplicativity of $||.||_Q$.

\begin{cor}
\label{cor:normed}
The Frobenius norm $||\, . \,||_Q$ is sub-multiplicative with respect to the geometric product $\diamond_h$:
\be
||\alpha \diamond_h\beta||_Q\leq ||\alpha||_Q||\beta||_Q \quad \forall \alpha,\beta\in \wedge V^\ast 
\ee
Hence the triplet $(\wedge V^\ast, \diamond_h, ||~||_Q)$ is a normed algebra which satisfies $||1||_Q=\sqrt{16}$. 
\end{cor}

\noindent
It turns out that the two scalar products $\langle~,~\rangle_h$ and $Q$ which we introduced on $\wedge V^\ast$ agree up to a factor. To show this, we need the following result:

\begin{lemma}
\label{lemma:Sh}
For any $\alpha\in \wedge^p V^\ast$ and $\beta\in \wedge^q V^\ast$, we have: 
\be
S(\alpha\diamond_h\beta)=16 (-1)^{\left[\frac{p}{2}\right]} \langle \alpha,\beta\rangle_h = 16 (-1)^{\frac{p(p-1)}{2}} \langle \alpha,\beta\rangle_h 
\ee
\end{lemma}

\begin{proof}
We have $\alpha\Delta_k\beta=0$ if $k > \min(p,q)$. Moreover, $(\alpha\Delta_k\beta)^{(0)}$ vanishes unless $2k= p+q$. The conditions $k=\frac{p+q}{2}$ and $k\leq \min(p,q)$ require $p=q=k$, so $\cS(\alpha\Delta_k^h\beta)=16 (\alpha\Delta_k^h\beta)^{(0)}$ vanishes except in this case. Using the expansion \eqref{DiamondExp}, this gives: 
\be
\cS(\alpha\diamond_h \beta)=16 \,\delta_{p,q} (-1)^{\left[\frac{p+1}{2}\right]+p^2} \alpha\Delta_p \beta=
16 (-1)^{\left[\frac{p+1}{2}\right]+p^2} \langle \alpha,\beta\rangle_h~,
\ee
where in the second equality we used the fact that $\langle \alpha,\beta\rangle_h$ vanishes unless $p=q$. The conclusion follows by noticing that: 
\be
(-1)^{\left[\frac{p+1}{2}\right]+p^2}=(-1)^{\left[\frac{p+1}{2}\right]-p}=(-1)^{\frac{p(p+1)}{2}-p}=(-1)^{\frac{p(p-1)}{2}}=(-1)^{\left[\frac{p}{2}\right]}~,
\ee
where we twice used the identity:
\be
(-1)^{\left[\frac{m}{2}\right]}=(-1)^{\frac{m(m-1)}{2}}\quad \forall m\in \Z 
\ee
\end{proof}

\begin{prop}
\label{prop:Qh}
The scalar products $\frac{1}{16}Q$ and $\langle~,~\rangle_h$ coincide on $\wedge V^\ast$. Hence for all $\alpha,\beta\in \wedge V^\ast$, we have: 
\ben
\label{hS}
\langle \alpha,\beta\rangle_h=\frac{1}{16} Q(\alpha,\beta)=\frac{1}{16}\cS(\tau(\alpha)\diamond_h \beta)=\frac{1}{16}\cS(\alpha\diamond_h \tau(\beta))
\een
In particular, we have:
\be
|\alpha|_h=\frac{1}{\sqrt{16}}||\alpha||_Q\quad \forall \alpha\in \wedge V^\ast
\ee
\end{prop}

\begin{proof}
For any $\alpha\in \wedge^p V^\ast$ and $\beta\in \wedge^q V^\ast$, we have: 
\be
Q(\alpha,\beta)=\cS(\tau(\alpha)\diamond_h\beta)=(-1)^{\frac{p(p-1)}{2}} \cS(\alpha\diamond_h\beta)=16\langle \alpha,\beta\rangle_h
\ee
where we used Lemma \ref{lemma:Sh}. The conclusion follows from bilinearity of $Q$ and $\langle~,~\rangle_h$. 
\end{proof}

\noindent In particular, we have: 
\be
|\alpha|_h^2=\frac{1}{16}\cS(\tau(\alpha)\diamond_h\alpha)=\frac{1}{16}\tr(\Psi_h^\gamma(\alpha)^t\circ \Psi_h^\gamma(\alpha))\quad \forall \alpha\in \wedge V^\ast 
\ee
Proposition \ref{prop:Qh} and Corollary \ref{cor:normed} imply: 

\begin{cor}
\label{cor:normdiamond}
The triplet $(\wedge V^\ast,\diamond_h, \sqrt{16} |\,\cdot\,|_h)$ is a normed algebra. Hence for any $\alpha,\beta\in \wedge V^\ast$, we have: 
\ben
\label{normdiamond}
|\alpha\diamond_h\beta|_h\leq \sqrt{16} |\alpha|_h|\beta|_h 
\een
\end{cor}

\noindent
The following lemma will be used in later computations.
\begin{lemma}
\label{lemma:cyc4form}
For any 4-forms $\omega_1,\omega_2,\omega_3\in \wedge^4 V^\ast$, we have: 
\ben
\label{cyc4form}
\langle \omega_1\Delta_2^h\omega_2,\omega_3\rangle_h=\langle \omega_1\diamond_h\omega_2,\omega_3\rangle_h=\langle \omega_2\diamond_h\omega_3,\omega_1\rangle_h=
\langle \omega_2\Delta_2^h\omega_3,\omega_1\rangle_h 
\een
Moreover, any four-form $\omega\in \wedge^4 V^\ast$ satisfies: 
\ben
\label{omegadiamondomega}
\omega\diamond_h\omega = |\omega|^2_h- \omega \Delta^h_2 \omega +\omega\wedge \omega\in \R\oplus \wedge^4 V^\ast\oplus \wedge^8 V^\ast 
\een
\end{lemma}

\begin{proof}
By Proposition \ref{prop:Qh}, we have: 
\be
\langle \omega_1\diamond_h\omega_2,\omega_3\rangle_h=\frac{1}{16}\cS(\omega_1\diamond_h\omega_2\diamond_h\omega_3)=\frac{1}{16}\cS(\omega_2\diamond_h\omega_3\diamond_h\omega_1)=\langle \omega_2\diamond_h\omega_3,\omega_1\rangle_h 
\ee
where we used the cyclicity of the K\"{a}hler-Atiyah trace and the fact that $\tau(\omega_i)=\omega_i$ since $\omega_i$ are four-forms. This gives \eqref{cyc4form}, where the leftmost and rightmost equalities follow from the expansion of the geometric product into generalized products and the 
fact that the scalar product $\langle~,~\rangle_h$ is block-diagonal with respect to the rank-decomposition of $\wedge V^\ast$. For any four-form $\omega\in \wedge^4 V^\ast$, relation \eqref{gradedasym} gives $\omega\Delta_k^h\omega=(-1)^k\omega\Delta_k^h\omega$, which implies that $\omega\Delta_k^h\omega$ vanishes for odd $k$. Hence the expansion \eqref{DiamondExp} of $\omega \diamond_h\omega$ reduces to \eqref{omegadiamondomega} upon using the identities: 
\be
\omega\Delta_4^h\omega=|\omega|_h^2~~\mathrm{and}~~\omega\Delta_0^h\omega=\omega\wedge\omega~ 
\ee
\end{proof}

\noindent Submultiplicativity of the Frobenius norm $||\,\cdot \,||_Q=\sqrt{16}||\,\cdot\,||_h$ implies the following inequality, which will be important later. 

\begin{prop}
\label{prop:bound}
For any nonzero self-dual four-form $\omega\in \wedge^4_+ V^\ast \setminus \{0\}$, we have: 
\ben
\label{bound}
|\omega\Delta_2^h\omega|_h < \sqrt{14} |\omega|_h^2 
\een
\end{prop}

\begin{proof}
Since $\omega$ is a four-form, we have $\tau(\omega)=\omega$ and hence the operator  $A:=\Psi_h^\gamma(\omega)\in \End(\Sigma)$ is $\cB$-symmetric. 
By Corollary \ref{cor:normdiamond}, we have: 
\ben
\label{e1}
|\omega\diamond_h\omega|^2_h\leq 16 |\omega|_h^4~,
\een
which is equivalent with the following inequality for the symmetric operator $A$: 
\ben
\label{e2}
\tr(A^4)\leq \tr(A^2)^2 
\een
On the other hand,  Lemma \ref{lemma:cyc4form} gives: 
\be
\omega\diamond_h\omega=|\omega|^2_h- \omega \Delta^h_2 \omega +|\omega|^2_h\nu_h~,
\ee
where we used self-duality of $\omega$ to simplify the last term. Since the terms in this expansion have different ranks, they are mutually orthogonal 
with respect to the scalar product $\langle~,~\rangle_h$ and hence:
\be
|\omega\diamond_h\omega|_h^2=2 |\omega|^4_h+| \omega \Delta^h_2 \omega|^2_h~,
\ee
where we used $|\nu_h|_h=1$. Using this in \eqref{e1} gives: 
\be
| \omega \Delta^h_2 \omega|_h^2\leq 14 |\omega|_h^4~,
\ee
which is equivalent with \eqref{bound}. Equality in \eqref{bound} holds if and only if it is achieved in \eqref{e1}, which is equivalent with \eqref{e2}. To understand when this happens, let $\lambda_1,\ldots, \lambda_{16}$ be the (real) eigenvalues of $A$ considered with multiplicity (hence some or all of these eigenvalues may be equal). Then: 
\be
\tr(A^4)=\sum_{i=1}^{16} \lambda_i^4~~\mathrm{and}~~\tr(A^2)^2=\left(\sum_{i=1}^{16} \lambda_i^2\right)^2=\sum_{i=1}^{16} \lambda_i^4+
2\sum_{1\leq i<j\leq 16} \lambda_i^2\lambda_j^2 
\ee
Hence $\tr(A^4)$ equals $\tr(A^2)^2$ if and only if:
\be
\sum_{1\leq i<j\leq 16} \lambda_i^2\lambda_j^2=0
\ee
which happens if and only if $\lambda_i\lambda_j=0$ for all $i\neq j$, which means that at most one of $\lambda_j$ can be nonzero. This happens if and only if $A=\lambda P$ where $\lambda=\tr(A)$ and $P$ is a $\cB$-orthogonal projector on a one-dimensional subspace of $\Sigma$. Since $A$ is $\cB$-symmetric, this amounts to the condition $A^2=\tr(A) A$, which is equivalent with $\omega\diamond_h\omega=\cS(\omega)\omega$ that is:
\be
|\omega|_h^2-\omega\Delta_2^h\omega+|\omega|_h^2\nu_h=16\omega^{(0)}\omega
\ee
which cannot be satisfied when $\omega\neq 0$ since the right-hand side is a form of degree four.
\end{proof}


\subsection{The spinorial description of $\Spin(7)$ structures}


Since $\nu_h$ squares to the identity in the K\"ahler-Atiyah algebra, we have $\Psi^{\gamma}_h(\nu_h)^2=\id_\Sigma$ and the vector space $\Sigma$ splits as a $\cB$-orthogonal direct sum $\Sigma = \Sigma^{+}\oplus \Sigma^{-}$, where $\Sigma^\pm$ are the eigenspaces of $\gamma_h(\nu_h)$ corresponding to the eigenvalues $\pm 1$. The subspaces define the \emph{chiral} irreducible representations of the even Clifford subalgebra $\Cl^\even(V^\ast,h^\ast)$ of chirality $\pm 1$, respectively. The spin group $\Spin(V^\ast,h^\ast)\subset \Cl^\even(V^\ast,h^\ast)$ naturally acts on $\Sigma^\pm$ through the representation induced by $\gamma$. The stabilizer of any nonzero chiral spinor is a $\Spin(7)$ of subgroup of $\Spin(V^{\ast},h^{\ast})\simeq \Spin(8)$. There exists two conjugacy orbits of such subgroups, which correspond respectively to the stabilizers of nonzero chiral spinors of positive and negative chirality and are called the $\Spin(7)_\pm$ conjugacy orbits of $\Spin(V^\ast,h^\ast)$. The later project onto the $\Spin(7)_\pm$ conjugacy classes of $\SO(V^\ast,h^\ast)\simeq \SO(V,h)$ through the double covering morphism: 
\ben
\label{lambda}
\lambda:\Spin(V,h)\rightarrow \SO(V,h)  
\een
Since the stabilizer of a chiral spinor depends only on its homothety class, this gives a bijection between the $\Spin(7)_\pm$ subgroups of $\Spin(V^\ast,h^\ast)$ and the real projective space $\P(\Sigma^\pm)$. Let  $(e^1 , \hdots , e^8 )$ is any orthonormal basis  of $(V^{\ast},h^{\ast})$. It is well-known that a nonzero four-form $\Phi\in \wedge^4 V^\ast$ is a conformal 
$\Spin(7)_\pm$ form on $(V,h)$ if and only if there exists a nonzero chiral spinor $\xi\in\Sigma^\pm\setminus\{0\}$ such that:
\ben
\label{PhiSpinor}
\Phi=\sum_{1\leq i_1 < \cdots < i_4\leq 8} \cB(\gamma_h(e^{i_1})\cdots \gamma_h(e^{i_4})\xi \, , \, \xi)\, e^{i_1}\wedge \cdots \wedge e^{i_4} 
\een
The $\Spin(7)_\pm$ stabilizer of this form in $\SO(V,h)$ coincides with the image of the $\Spin(7)_\pm$ stabilizer of $\xi$ through the double covering morphism \eqref{lambda}. The chiral spinor $\xi$ with this property is determined by $\Phi$ up to sign. It is convenient to use the notation $\gamma^i:=\gamma_h(e^i)\in \End(\Sigma)$ where $i\in \{1,\ldots, 8\}$. These operators are $\cB$-symmetric since $\tau(e^i)=e^i$. Moreover, they satisfy\footnote{We write the composition $\circ$ of operators as juxtaposition for ease of notation.}: 
\ben
\label{acomm}
\gamma_i\gamma_j+\gamma_j\gamma_i=2\delta_{ij}\id_\Sigma~~\quad~~\forall i,j\in \{1,\ldots,8\} 
\een
and: 
\ben
\label{gammavol}
\gamma^i\gamma_h(\nu_h)=-\gamma_h(\nu_h)\gamma^i~~\quad~~\forall i\in \{1,\ldots, 8\}~,
\een
where the last relation follows from \eqref{acomm} and the identity $\nu_h=e^1\wedge \ldots \wedge e^8=e^1\diamond_h\ldots \diamond_h e^8$, which implies: 
\be
\gamma_h(\nu_h)=\gamma^1\ldots \gamma^8 
\ee
Since $\Sigma^\mu$ is the eigenspace of $\gamma_h(\nu_h)$ which corresponds to eigenvalue $\mu$ and $\gamma^i$ square to the identity, relation \eqref{gammavol} implies: 
\ben
\label{gammaSigma}
\gamma^i(\Sigma^\pm)=\Sigma^\mp 
\een

\subsection{Signed spinor squares}

Let $\kappa\in \{-1,1\}$ be a sign factor. The \emph{$\kappa$-signed spinor squaring map} of the triplet $(\Sigma,\gamma,\cB)$ is the bilinear map $\cE^{\kappa}_{\gamma}: \Sigma \rightarrow \wedge V^{\ast}$ defined through (see \cite{CLS}):
\be
\cE^{\kappa}_{\gamma}(\xi)= \kappa\,(\Psi^{\gamma}_h)^{-1}(\xi\otimes \xi^{\ast}) \quad \forall \xi\in \Sigma 
\ee
Here $\xi^\ast\in \Sigma^\ast$ is the Riesz dual of $\xi\in \Sigma$ with respect to the scalar product $\cB$ and we identify $\xi\otimes \xi^\ast\in \Sigma\otimes \Sigma^\ast$ with an endomorphism of $\Sigma$ using the natural isomorphism $\Sigma\otimes \Sigma^\ast\simeq \End(\Sigma)$. We say that a polyform $\alpha\in \wedge V^{\ast}$ is the \emph{signed square} of a spinor if $\alpha\in \Im(\cE^{+}_{\gamma})\cup \Im(\cE^{-}_{\gamma})$. In this case, we say that $\alpha$ is \emph{$\kappa$-signed} if $\alpha\in \Im(\cE^\kappa_{\gamma})$. The maps $\cE_\gamma^+$ and $\cE_\gamma^-$ are called respectively the positive and negative spinor squaring maps.  

\begin{remark}
As shown in, \cite[Proposition 3.22]{CLS}, the $\kappa$-signed square of a spinor $\xi\in \Sigma$ can be expressed as: 
\ben
\label{eq:alphaexp}
\cE_\gamma^\kappa(\xi) = \frac{\kappa}{16} \sum_{k=0}^8 \sum_{1\leq i_1 < \cdots < i_k\leq 8} \cB(\gamma_h(e^{i_k})\cdots \gamma_h(e^{i_1})\xi,\xi)\, e^{i_1}\wedge \cdots \wedge e^{i_k}~,
\een
where $(e^1 , \hdots , e^8)$ is any orthonormal basis of $(V^{\ast},h^{\ast})$. See also \cite[Eq. (5.25)]{ga1} or \cite[Eq. (2.2)]{gf}. 
\end{remark}

\begin{lemma}
\label{lemma:norm4}
For any $\xi\in \Sigma$, we have: 
\ben
\label{norm4}
\cB(\xi,\xi)^2=\frac{1}{16} \sum_{k=0}^8 \sum_{1\leq i_1 < \cdots < i_k\leq 8} \cB(\xi, \gamma_h(e^{i_1})\cdots \gamma_h(e^{i_k})\xi)^2 
\een
\end{lemma}

\begin{proof}
Applying $\Psi_h^\gamma$ to both sides of \eqref{eq:alphaexp} gives: 
\be
\xi\otimes \xi^\ast=\frac{1}{16} \sum_{k=0}^8 \sum_{1\leq i_1 < \cdots < i_k\leq 8} \cB(\gamma^{i_k}\cdots \gamma^{i_1} \xi, \xi )  \gamma^{i_1}\cdots \gamma^{i_k} 
\ee
Evaluating this on $\xi$, we obtain: 
\be
\cB(\xi,\xi)\xi=\frac{1}{16} \sum_{k=0}^8 \sum_{1\leq i_1 < \cdots < i_k\leq 8} \cB(\gamma^{i_k}\cdots \gamma^{i_1} \xi, \xi )  \gamma^{i_1}\cdots \gamma^{i_k}\xi
\ee
which implies:
\be
\cB(\xi,\xi)^2=\frac{1}{16} \sum_{k=0}^8 \sum_{1\leq i_1 < \cdots < i_k\leq 8} \cB(\gamma^{i_k}\cdots \gamma^{i_1} \xi, \xi ) \cB(\xi, \gamma^{i_1}\cdots \gamma^{i_k}\xi)
\ee
and reduces to \eqref{norm4} upon using the fact that $\gamma^i=\gamma_h(e^i)$ are $\cB$-symmetric.
\end{proof}

\begin{prop}\cite[Proposition 3.19]{CLS}
\label{prop:equivariancelambda}
For any $\kappa\in \{-1,1\}$, the $\kappa$-signed spinor squaring map $\cE^{\kappa}_{\gamma}\colon \Sigma \to (\wedge V^{\ast},\diamond_h)$ is equivariant with respect to the double covering morphism $\lambda\colon \Spin(V^{\ast},h^{\ast}) \to \SO(V^{\ast},h^{\ast})$, that is:
\be
\cE^{\kappa}_{\gamma}(\gamma(u)(\xi)) = \lambda(u)(\cE^{\kappa}_{\gamma}(\xi))\quad \forall\,\, u\in \Spin(V^{\ast},h^{\ast})\quad \forall \xi\in \Sigma~,
\ee
where in the side we use the natural action of $\lambda(u)\in\SO(V^{\ast},h^{\ast})$ on $\wedge V^{\ast}$.
\end{prop}

\noindent 
The starting point of our investigation is the following result, which is a direct consequence of Corollary 3.28 in \cite{CLS}.

\begin{prop}
\label{prop:squarespinor} 
A nonzero polyform $\alpha\in \wedge V^{\ast}$ is a signed square of a nonzero chiral spinor $\xi\in \Sigma^{\mu}$ of chirality $\mu\in \{-1,1\}$ (i.e. we have $\alpha=\cE_\gamma^\kappa(\xi)$ for some $\kappa\in \{-1,1\}$) if and only if:
\begin{equation}
\label{eq:algebraiceqs1}
\alpha\diamond_h \alpha = 16\, \alpha^{(0)}\, \alpha \, , \qquad \tau(\alpha) = \alpha\, , \qquad \nu_h\diamond_h\alpha= \mu\, \alpha~ ,
\end{equation}
where $\alpha^{(0)}$ denotes the zero-rank component of $\alpha$. In this case, we have:
\ben
\label{cBalpha0}
\cS(\alpha)=16\, \alpha^{(0)} = \kappa \cB(\xi,\xi)~,
\een
and hence $\alpha^{(0)}\neq 0$ and $\kappa=\sign\, \alpha^{(0)}=\sign\, \cS(\alpha)$ and $\cB(\xi,\xi)=16 |\alpha^{(0)}|=|\cS(\alpha)|$. 
\end{prop}

\noindent
Relation \eqref{cBalpha0} from equation \eqref{eq:alphaexp}. This relation shows that the spinor $\xi$ has unit $\cB$-norm if and only if $|\alpha^{(0)}|=\frac{1}{16}$, i.e. $|\cS(\alpha)|=1$. In that case, we say that $\xi$ is \emph{normalized} or of \emph{unit norm}.


\subsection{The algebraic characterization of conformal $\Spin(7)$ forms}


In this subsection, we present an algebraic characterization of conformal $\Spin(7)_\pm$ forms which follows from the description of the signed square of a real irreducible chiral spinor given in Proposition \ref{prop:squarespinor}. Before proceeding, let us explain the idea behind this approach. As shown in \cite{CLS}, the restriction of $\cE^\kappa_\gamma$ to any of the sets $\Sigma^+\setminus \{0\}$ or $\Sigma^-\setminus \{0\}$ is two to one for any $\kappa\in \{-1,1\}$. More precisely, a nonzero chiral spinor $\xi\in \Sigma^\mu\setminus \{0\}$ of fixed chirality $\mu\in \{-1,1\}$ which satisfies $\cE^\kappa_\gamma(\xi)=\alpha$ is determined by the polyform $\alpha$ up to sign. The stabilizer of $\xi$ is a $\Spin(7)$ subgroup of $\Spin(V^\ast,h^\ast)$ (which projects to a $\Spin(7)$ subgroup of $\SO(V^\ast,h^\ast)$) and giving such a spinor is equivalent to giving conformal $\Spin(7)_\mu$ form on $(V,h)$. This form determines the chiral spinor up to sign. Since signed spinor squaring maps are two-to-one when restricted to the sets of nonzero chiral spinors and since the group $\Spin(7)$ is simply connected, Proposition \ref{prop:equivariancelambda} implies that the stabilizer of $\cE^\kappa_\gamma(\xi)$ is isomorphic to $\Spin(7)$ for any nonzero chiral spinor. Hence specifying a conformal $\Spin(7)_\pm$ form on $(V,h)$ is equivalent to specifying the $\kappa$-signed square polyform $\alpha=\cE_\gamma^\kappa(\xi) \in \wedge V^\ast$ of a nonzero chiral spinor of chirality $\pm 1$ and $\xi$ is determined by its polyform square up to sign. In turn, such a polyform is characterized by the algebraic conditions in Proposition \ref{prop:squarespinor}, which yields the algebraic characterization of metric $\Spin(7)_\pm$ forms on $(V,h)$ that we will explore for the remaining of this note. To give this characterization explicitly, we need the following: 

\begin{lemma}
\label{lemma:squarespinor} 	
Let $\kappa\in \{-1,1\}$ be a sign factor. A non-zero polyform $\alpha\in \wedge V^\ast$ is the $\kappa$-signed square of a non-zero chiral spinor $\xi\in \Sigma^\mu$ of chirality $\mu\in \{-1,1\}$ if and only if it has the form:
\ben
\label{eq:squareformalgebraic}
\alpha =\frac{\kappa}{16}\left[\frac{1}{\sqrt{14}}|\Phi|_h +\Phi + \frac{\mu}{\sqrt{14}}|\Phi|_h \nu_h\right]~,
\een
where $\Phi \in \wedge^4 V^{\ast}$ is a uniquely determined four-form which satisfies the condition $\ast_h \Phi=\mu \Phi$ as well as the algebraic equation: 
\ben
\label{eq:quadsys}
\sqrt{14} \Phi\Delta_2^h \Phi + 12 |\Phi|_h \Phi = 0 
\een
In this case, we have:
\ben
\label{xiNormPhi}
\cB(\xi,\xi)=\frac{|\Phi|_h}{\sqrt{14}}
\een
and hence $\xi$ has unit $\cB$-norm if and only if $|\Phi|_h = \sqrt{14}$. The nonzero chiral spinor $\xi\in \Sigma^\mu$ is determined by the polyform $\alpha$ up sign. 
\end{lemma}

\begin{remark}
Notice that we can replace $\Phi$ by $-\Phi$ in \eqref{eq:squareformalgebraic} provided that we also do so in \eqref{eq:quadsys}. This changes the 
sign of the middle term in  \eqref{eq:squareformalgebraic}  and of the second term in the left hand side of \eqref{eq:quadsys}. Which of the signs we choose 
is a matter of convention, provided that the signs of these terms in  \eqref{eq:squareformalgebraic} and \eqref{eq:quadsys} are changed simultaneously. 
\end{remark}

\begin{proof} 
By Proposition \ref{prop:squarespinor}, a polyform $\alpha \in \wedge V^{\ast}$ is the $\kappa$-signed square of a chiral spinor $\xi$ of chirality $\mu$ if and only if $\kappa\alpha^{(0)}>0$ and equations \eqref{eq:algebraiceqs1} are satisfied. The general solution of the second and third equations in \eqref{eq:algebraiceqs1} can be written as:
\ben
\label{alphasol}
\alpha = \frac{\kappa}{16}(c + \Phi + \mu c\nu_h)~,
\een
where $\Phi$ is a four-form which satisfies $\ast_h \Phi=\mu \Phi$ and $c\in\mathbb{R}$ is a constant. This constant must be positive since
$\alpha^{(0)}=\frac{\kappa c}{16}$. Plugging \eqref{alphasol} into the first equation of \eqref{eq:algebraiceqs1} gives:
\ben
\label{eq:QuadraticPhiLambda}
\Phi\diamond_h \Phi = 12\,c\,\Phi+ 14\, c^2 (1+\mu\nu_h)~,
\een
where we used the last equation in \eqref{eq:algebraiceqs1} and the fact that $\nu_h$ squares to $1$ in the \KA algebra. Expanding the geometric product $\diamond_h$ gives (see \eqref{omegadiamondomega}):
\ben
 \label{PhiSquareExp}
\Phi\diamond_h\Phi = \vert\Phi\vert_h^2- \Phi \Delta^h_2 \Phi +\Phi\wedge \Phi \in \R\oplus \wedge^4 V^\ast\oplus \wedge^8 V^\ast 
\een
Separating degrees, this shows that \eqref{eq:QuadraticPhiLambda} is equivalent with the conditions: 
\ben
\label{conds}
|\Phi|_h^2=14 c^2 \, , \qquad \Phi\Delta_2^h \Phi + 12 c \Phi = 0\, , \qquad \Phi\wedge \Phi = 14\, c^2 \mu \nu_h 
\een
The first and last of these conditions are equivalent since $\ast_h\Phi=\mu \Phi$ , and give:
\ben
c = \frac{\vert\Phi\vert_h}{\sqrt{14}}  
\een
Substituting this into the middle equation of \eqref{conds} gives \eqref{eq:quadsys}. Relation \eqref{alphasol} gives $\cS(\alpha)=\kappa c$, which implies $\cB(\xi,\xi)=c=\frac{\vert\Phi\vert_h}{\sqrt{14}}$ by Proposition \ref{prop:squarespinor}. Substituting the value of $c$ into \eqref{alphasol} gives \eqref{eq:squareformalgebraic}.
\end{proof}

\noindent
The remaining components of the polyform $\kappa\alpha$ appearing in \eqref{eq:squareformalgebraic} are uniquely determined by the norm of this four-form and by the chirality $\mu$ of $\xi$. Relations \eqref{eq:alphaexp} and \eqref{eq:squareformalgebraic} imply that $\Phi$ is given by \eqref{PhiSpinor} and hence is a conformal $\Spin(7)_\mu$ form on $(V,h)$.

\begin{lemma}
\label{lemma:ClassicalForm}
Let $\xi\in \Sigma^\mu$ be a nonzero chiral spinor of chirality $\mu\in \{-1,1\}$ and let $\Phi$ be the self-dual four-form given in \eqref{PhiSpinor}, where 
 $(e^1 , \hdots , e^8 )$ is any orthonormal basis  of $(V^{\ast},h^{\ast})$. Then the following statements hold: 
\begin{enumerate}
\item We have: 
\ben
\label{normxi}
\cB(\xi,\xi)=\frac{|\Phi|_h}{\sqrt{14}} 
\een
\item The $\kappa$-signed square of $\xi$ is given by: 
\ben
\label{sqxi}
\cE_\gamma^\kappa(\xi)=\frac{\kappa}{16}\left[\frac{1}{\sqrt{14}}\vert\Phi\vert_h +\Phi + \frac{\mu}{\sqrt{14}}|\Phi|_h\nu_h\right] 
\een
\end{enumerate}
\end{lemma}

\begin{proof}
The quantity $\cB(\xi, \gamma^{i_1}\cdots \gamma^{i_k}\xi)$ with $i_1<\ldots< i_4$ vanishes by $\cB$-symmetry of $\gamma^i$ unless $\frac{k(k-1)}{2}$ is even, which requires $k\in \{0,1, 4, 5, 8\}$. On the other had, this quantity vanishes when $k$ is odd since in that case $\gamma^{i_1}\cdots \gamma^{i_k}\xi$ 
has different chirality from $\xi$ (see \eqref{gammaSigma}) and since the spaces $\Sigma^+$ and $\Sigma^-$ are $\cB$-orthogonal.  Hence \eqref{norm4} reduces to:
\be
\cB(\xi,\xi)^2=\frac{1}{16} (2\cB(\xi,\xi)^2 +|\Phi|_h^2)~~,
\ee
which gives \eqref{normxi}. On the other hand, the expansion \eqref{eq:alphaexp} reduces to: 
\be
\cE_\gamma^\kappa(\xi)=\frac{\kappa}{16}\left[\cB(\xi,\xi)+\Phi+\mu\cB(\xi,\xi)\nu_h\right]~,
\ee
where we used the relation $\gamma_h(\nu_h)\xi=\mu \xi$. Combing this with \eqref{normxi} gives \eqref{sqxi}.
\end{proof}

\begin{thm}
\label{thm:Spin7algebraic} 
 The following statements are equivalent for a nonzero self-dual four-form $\Phi\in \wedge^4_{+} V^{\ast}$: 
 \begin{enumerate}[(a)]
\item $\Phi$ is a conformal $\Spin(7)_+$ form on $(V,h)$.
\item The polyform:
\ben
\label{eq:alphapolyformform}
\alpha =\frac{1}{16}\left[\frac{1}{\sqrt{14}}\vert\Phi\vert_h +\Phi + \frac{1}{\sqrt{14}}|\Phi|_h\nu_h\right]
\een
is the positive square of a nonzero positive chirality spinor. 
\item $\Phi$ satisfies the following algebraic equation:
\ben
\label{eq:algebraicconditionII}
\sqrt{14} \Phi\Delta_2^h \Phi + 12 |\Phi|_h \Phi = 0 
\een
\end{enumerate}
In particular, there is a one-to-one correspondence between the set of conformal $\Spin(7)_+$ forms on $(V,h)$ and the set of nonzero self-dual solutions of equation \eqref{eq:algebraicconditionII}.
\end{thm}
 
\begin{proof} 
Suppose that $(a)$ holds. Then $\Phi$ is given by \eqref{PhiSpinor} for some non-zero positive chirality 
spinor $\xi$. Lemma \ref{lemma:ClassicalForm} implies that the polyform $\alpha:=\cE_\gamma^+(\xi)$ is given by \eqref{eq:alphapolyformform} with $\kappa=1$ and hence $(b)$ holds. Thus $(a)$ implies $(b)$. Lemma \ref{lemma:squarespinor} shows that $(b)$ implies $(c)$. Finally, let us assume that $(c)$ holds. Then Lemma \ref{lemma:squarespinor} shows the polyform $\alpha$ defined in terms of $\Phi$ by formula \eqref{eq:alphapolyformform} is the positive square of a positive chirality spinor $\xi$, and hence is given by \eqref{eq:alphaexp} with $\kappa=1$, so that $(a)$ holds. This shows that $(c)$ implies $(a)$ and we conclude. 
\end{proof}

\noindent
Theorem \ref{thm:Spin7algebraic} and Lemma \ref{lemma:conformalconstant} imply the following algebraic characterization of metric $\Spin(7)_+$ forms. 

\begin{cor}
\label{cor:Spin7metricalgebraic} 
A self-dual four-form $\Phi$ is a metric $\Spin(7)_+$ form on $(V,h)$ if and only if $|\Phi|_h = \sqrt{14}$ and $\Phi$ satisfies the equation:
\ben
\label{metricSpin7criterion}
\Phi\Delta^h_2 \Phi+12 \Phi=0 
\een 
\end{cor}

\noindent
Verifying that a self-dual four-form is a conformal $\Spin(7)_{+}$ form by traditional methods involves dealing with the cumbersome task of checking if there exists a basis of $V^{\ast}$ in which $\Phi$ is proportional to the expression given in equation \eqref{Phi_0}. Theorem \ref{thm:Spin7algebraic} provides a different criterion which can be used to verify if such a four-form is a conformal $\Spin(7)_{+}$ form: we only need to check if equation \eqref{eq:algebraicconditionII} is satisfied. This gives an intrinsic characterization of conformal $\Spin(7)_{+}$ forms that do not require the use of any privileged basis of $V^{\ast}$ and that we hope it can be useful for applications, see for instance \cite{fol1,fol2} for early applications of this framework. 

\begin{remark}
By Lemma \ref{lemma:squarespinor}, a conformal $\Spin(7)_+$ form on $(V,h)$ is metric if and only if the corresponding positive chirality spinor $\xi\in \Sigma^+$ (which is determined up to sign) has unit $\cB$-norm. Notice that equation \eqref{metricSpin7criterion} for metric $\Spin(7)_+$ forms is quadratic but inhomogeneous in $\Phi$, unlike equation \eqref{eq:algebraicconditionII} for conformal $\Spin(7)_+$ forms, which is not quadratic in $\Phi$ but is homogeneous of order two under rescaling  $\Phi$ by a positive constant. Also notice that $\Phi=0$ is a special solution of \eqref{eq:algebraicconditionII}, as expected.
\end{remark}

\noindent
The results above show that the cone $C_+(V,h)$ defined in \eqref{cC} coincides with the real algebraic set: 
\be
C_+(V,h)=\{\Phi\in \wedge_+^4 V^\ast \vert \sqrt{14} \Phi\Delta_2^h \Phi + 12|\Phi|_h \Phi = 0\}~,
\ee
while the set $\cC_+(V,h)$ of all conformal $\Spin(7)_+$ forms is the union of pointed cones obtained by removing the origin from $C_+(V,h)$. Notice that equation \eqref{eq:algebraicconditionII} is invariant under the action of $\SO(V,h)$ on $\wedge^4_+ V^\ast$, as are the two equations in \eqref{metricSpin7criterion}. Hence the sets $C_+(V,h)$ and $\cA_+(V,h)$ of solutions to these equations are also rotationally invariant. An immediate consequence of Corollary \ref{cor:Spin7metricalgebraic} is an algebraic realization of $\cA_+(V,h)$ inside $\mathbb{R}^{35}$.

\begin{cor}
The manifold $\cA_+(V,h)$ of metric $\Spin(7)_+$ forms on $(V,h)$ is diffeomorphic with the following real affine variety  given by quadratic equations inside the $35$-dimensional real vector space $\wedge^4_{+} V^{\ast}$ of self-dual four-forms:
\be
A_+(V,h) =\left\{ \Phi \in \wedge^4_{+} V^{\ast} \,\, |\,\, \Phi\Delta^h_2 \Phi+12 \Phi = 0 ~,~|\Phi|_h^2=14\right\}
\ee
\end{cor}

\noindent 
Since $\cA_+(V,h) \simeq \R\P^7$, this gives a real algebraic embedding of $\R\P^7$ into $\mathbb{R}^{35}$. Fix an orthonormal basis $(e^1,\hdots,e^8)$ of $(V^{\ast},h^{\ast})$ and use it to identify:
\be
V^{\ast} =\R^8\, , \qquad \wedge^4_{+} V^{\ast} = \R^{35} 
\ee

\noindent
Let $\rS^{34}_{\sqrt{14}}\subset \R^{35}$ be the sphere of radius $\sqrt{14}$ centered at the origin of $\mathbb{R}^{35}$, which we view as a real algebraic variety. Using the previous identifications, we have:
\be
A_+(\R^8,h_0) = \left\{ \Phi \in \rS^{34}_{\sqrt{14}} \,\, |\,\, \Phi\Delta^h_2 \Phi = -12 \Phi\right\}  
\ee
This realizes $\mathbb{RP}^7$ as a real algebraic variety given by quadratic equations inside $S^{34}_{\sqrt{14}}$.


\section{An algebraic potential for conformal \texorpdfstring{$\Spin(7)$}{Spin(7)} forms} 
\label{sec:Wh}


To describe the set of conformal $\Spin(7)_+$ forms on $(V,h)$, consider the following cubic function defined on the vector space of self-dual four-forms:
\ben
\label{eq:cubicfunction}
W_h:\wedge^4_{+} V^{\ast} \rightarrow \mathbb{R}\, ,\quad \Phi \mapsto  W_h(\Phi) := -\frac{\sqrt{14}}{3}\langle \Phi\diamond_h \Phi , \Phi \rangle_h + 4 \langle \Phi , \Phi \rangle^{\frac{3}{2}}_h=\frac{\sqrt{14}}{3}\langle \Phi\Delta_2^h\Phi , \Phi \rangle_h + 4 \langle \Phi , \Phi \rangle_h^{\frac{3}{2}} 
\een
The second equality above follows from \eqref{PhiSquareExp}, which implies: 
\ben
\label{PhiDiamondDelta}
\langle \Phi\diamond_h \Phi , \Phi\rangle_h=-\langle \Phi\Delta_2^h\Phi , \Phi \rangle_h
\een
since $\langle ~ ,~ \rangle_h $ is block-diagonal with respect to the rank decomposition of $\wedge V^\ast$. Using Proposition \ref{prop:Qh}, 
we can write $W_h$ as: 
\ben
\label{WS}
W_h(\Phi)=-\frac{\sqrt{14}}{48}\cS(\Phi\diamond_h \Phi\diamond_h\Phi) + \frac{1}{16} \cS(\Phi\diamond_h\Phi)^{\frac{3}{2}}~,
\een
where we noticed that $\tau(\Phi)=\Phi$ since $\Phi$ is a four-form. Using the rotational invariance of the K\"{a}hler-Atiyah trace and the rotational equivariance of the generalized products, relation \eqref{WS} implies that the potential $W_h:\wedge^4_+ V^\ast\rightarrow \R$ is invariant under the natural action of $\SO(V,h)$ on $\wedge^4_+ V^\ast$. Moreover, notice that $W_h$ is homogeneous of degree three under positive rescalings $\Phi\rightarrow \lambda \Phi$ ($\lambda>0$) of its argument and that we have $W_h(0)=0$. In particular, the restriction of $W_h$ to $\wedge^4_+ V^\ast$ descends to a section of the real line bundle $\cO(3)\rightarrow \P(\wedge^4_+ V^\ast)$.

\begin{remark}
\label{rem:WBasis}
Direct computation gives: 
\ben
\langle \Phi\Delta^h_2\Phi , \Phi \rangle_h=\frac{1}{8}\Phi_{i j k l}\Phi^{i j}_{\,\,\,\,\, m n} \Phi^{m n k l }~\mathrm{and}~\langle \Phi , \Phi \rangle_h=\frac{1}{24} \Phi_{i j k l}\Phi^{i j k l}~.
\een
Thus: 
\ben
W_h(\Phi)=\frac{\sqrt{14}}{24}\Phi_{i j k l}\Phi^{i j}_{\,\,\,\,\, m n} \Phi^{m n k l }+4\left(\frac{1}{24}\Phi_{i j k l}\Phi^{i j k l}\right)^{3/2}~,
\een
in any basis $(e^1,\ldots, e^8)$ of $(V,h)$. Notice that we use the determinant inner product of forms, with respect to which the volume form of $(V,h)$ has norm one. Also notice that $W_h$ is of class $C^2$ on $\wedge^4_+ V^\ast$, since the 3-rd power of the norm function $||\cdot||: \wedge^4_+ V^\ast\rightarrow \R$ is of class $C^2$ (and its first two differentials vanish at the origin). See \cite[Theorem 3.1]{RN}.
\end{remark}

\begin{prop}
\label{prop:d2Phi}
Let $\Phi\in \wedge^4_+ V^\ast$ be a self-dual four-form on $(V,h)$. For any $q\in \wedge^4_+ V^\ast$, we have: 
\begin{eqnarray}
& W_h(\Phi+q)=W_h(\Phi)-\sqrt{14}\left[\langle \Phi\diamond_h \Phi, q\rangle+
 \langle q\diamond_h q, \Phi\rangle_h+\frac{1}{3}\langle q\diamond_h q, q\rangle_h\right] \nonumber \\
& +4\left[\left(|\Phi|_h^2+|q|_h^2+2\langle \Phi,q\rangle_h\right)^{\frac{3}{2}}-|\Phi|_h^3\right] \label{WhExpansion}
\end{eqnarray}

\noindent
When $\Phi\neq 0$ and $|q|_h\ll 1$, the Taylor expansion of $W_h(\Phi)$ around $\Phi$ is: 
\ben
\label{WhExp}
W_h(\Phi+q)=W_h(\Phi)+\langle -\sqrt{14}\Phi\diamond_h\Phi+12 |\Phi|_h\Phi,q\rangle_h +
6|\Phi|_h|q|_h^2-\sqrt{14}\langle q\diamond_h q,\Phi\rangle_h+ 6\frac{\langle q,\Phi\rangle_h^2}{|\Phi|_h}+\O(\epsilon^3)
\een
for $\epsilon:=\frac{|q|_h}{|\Phi|_h}\ll 1$
\end{prop}

\begin{proof}
We have:
\be
W_h(\Phi+q)=-\frac{\sqrt{14}}{3}\langle (\Phi+q)\diamond_h (\Phi+q) , \Phi+q \rangle_h + 4 \langle \Phi +q, \Phi +q\rangle^{\frac{3}{2}}_h~,
\ee
Using  Lemma \ref{lemma:cyc4form}, we compute: 
\be
\langle (\Phi+q)\diamond_h (\Phi+q) , \Phi+q \rangle_h=\langle \Phi\diamond_h\Phi,\Phi\rangle_h+3\langle \Phi\diamond_h \Phi, q\rangle+
3 \langle q\diamond_h q, \Phi\rangle_h+\langle q\diamond_h q, q\rangle_h
\ee
and: 
\be
\langle \Phi +q, \Phi +q\rangle_h=|\Phi|_h^2+|q|_h^2+2\langle \Phi,q\rangle_h 
\ee
Combining these relations gives \eqref{WhExpansion}. When $\epsilon=\frac{|q|_h}{|\Phi|_h}\ll 1$, we have:
\beqa
& \left(|\Phi|_h^2+|q|_h^2+2\langle \Phi,q\rangle_h\right)^{\frac{3}{2}} = |\Phi|_h^3\left(1+\frac{|q|_h^2+2\langle \Phi,q\rangle_h}{|\Phi|_h^2}\right)^{\frac{3}{2}}\nn\\
& = |\Phi|_h^3 \left[1+\frac{3}{2}\frac{|q|_h^2+2\langle \Phi,q\rangle_h}{|\Phi|_h^2}+ \frac{3}{8}\left(\frac{|q|_h^2+2\langle \Phi,q\rangle_h}{|\Phi|_h^2}\right)^2 \right]+\O\left(\epsilon^3\right) 
\eeqa
Using this in \eqref{WhExpansion} gives:
\beqa
& W_h(\Phi+q) - W_h(\Phi) =  \\
& = -\sqrt{14}\left[\langle \Phi\diamond_h \Phi, q\rangle+
 \langle q\diamond_h q, \Phi\rangle_h\right]+6|\Phi|_h (|q|_h^2+2\langle \Phi,q\rangle_h) + \frac{3}{2}\frac{(|q|_h^2+2\langle \Phi,q\rangle_h)^2}{|\Phi|_h}+\O(\epsilon^3) \\
& = W_h(\Phi)+\langle -\sqrt{14}\Phi\diamond_h\Phi+12 |\Phi|_h\Phi,q\rangle_h + 6|\Phi|_h|q|_h^2-\sqrt{14}\langle q\diamond_h q,\Phi\rangle_h+ 6\frac{\langle q,\Phi\rangle_h^2}{|\Phi|_h}+\O(\epsilon^3) 
\eeqa
\end{proof}

\noindent 
Consider a nonzero self-dual four-form $\Phi\in \wedge_{+}^4 V^{\ast}\setminus \{0\}$. We denote by:
\be
\dd_{\Phi} W_h \colon \wedge^4_{+} V^{\ast} \to \R\, , \qquad q \mapsto (\dd_{\Phi} W_h) (q)
\ee
the differential of $W_h$ computed at $\Phi$ and by:
\be
\dd^2_{\Phi}W_h \colon \wedge^4_{+} V^{\ast} \odot \wedge^4_{+} V^{\ast}\to \R
\ee
the second differential of $W_h$ at $\Phi$. The latter coincides with the Euclidean Hessian of $W_h$ at $\Phi$ computed relative to the scalar product $\langle~,~\rangle_h$. 

\begin{cor}
\label{cor:dWh}
For any nonzero self-dual four-form $\Phi\in \wedge^4_+ V^\ast \setminus \{0\}$, we have: 
\ben
\label{dWh}
(\dd_{\Phi} W_h) (q)=\langle\sqrt{14}\Phi\Delta_2^h\Phi + 12 |\Phi|_h \Phi, q\rangle_h \quad \forall q\in \wedge^4_+ V^\ast
\een
and:
\ben
\label{ddWh}
(\dd^2_{\Phi}W_h) (q_1 , q_2) =  \sqrt{14}\,\langle  q_1\Delta_2^h q_2+q_2\Delta_2^h q_1 , \Phi \rangle_h + 12 \vert \Phi\vert_h  \langle   q_1, q_2 \rangle_h + \frac{12}{\vert\Phi\vert_h}  \langle q_1 , \Phi\rangle_h \langle \Phi, q_2 \rangle_h  
\een

\noindent
for every $q_1, q_2\in \wedge^4_{+} V^{\ast}$.
\end{cor}

\begin{proof}
Follows immediately from \eqref{WhExp} by polarization.
\end{proof}

\begin{remark}
Using \cite[Theorem 3.1]{RN}, one easily checks that $(\dd^2 W_h)(0)=0$. This also follows from the fact that $W_h$ is of class $C^2$. The Cauchy–Schwarz inequality gives:
\be
\frac{1}{\vert\Phi\vert_h} \vert \langle q_1 , \Phi\rangle_h \langle \Phi, q_2 \rangle_h| \leq |\Phi|_h |q_1|_h |q_2|_h 
\ee
and hence the $\lim_{\Phi \to 0} \dd^2_{\Phi}W_h$ exists and equals zero. 
\end{remark}

\begin{thm}
\label{thm:PotentialVh}
A self-dual four-form $\Phi\in \wedge^4_{+} V^{\ast}\setminus \{0\}$ is a conformal $\Spin(7)_+$ form on $(V,h)$ if and only if it is a critical point of the function $W_h$. In this case, we have: 
\be
W_h(\Phi)=0 
\ee
\end{thm}

\begin{proof}
Since $\langle~,~\rangle_h$ restricts to a scalar product on $\wedge^4_+ V^\ast$, Corollary \eqref{cor:dWh} implies that $\dd_\Phi W_h$ vanishes if and only if $\Phi$ satisfies equation \eqref{eq:algebraicconditionII}, which by Theorem \ref{thm:Spin7algebraic} happens if and only if $\Phi$  is a conformal $\Spin(7)_+$ form on $(V,h)$. In this case, we have:
\be
W_h(\Phi)  = \frac{\sqrt{14}}{3}\langle \Phi\Delta_2^h \Phi , \Phi \rangle_h + 4 |\Phi|_h^3 =-4|\Phi|^3_h +4 |\Phi|_h^3 =0
\ee
\end{proof}

\subsection{Representation-theoretic analysis of the Hessian of $W_h$}

Let $\Phi\in \wedge^4_+ V^\ast$ be a conformal $\Spin(7)_+$ form on $(V,h)$. We denote the induced metric of $\Phi$ by $h_{\Phi}$. Recall that $\Phi$ is a \emph{metric} $\Spin(7)_+$ form with respect to $h_{\Phi}$. Consider the orthogonal decomposition of $\wedge^4_{+} V^{\ast}$ into irreducible representations under the linear action of the $\Spin(7)$ stabilizer of $\Phi$ which is obtained by restricting the obvious action of $\SO(V,h)$ on $\wedge^4_{+} V^{\ast}$ (see, for example, \cite{KarigiannisFlows}):
\ben
\label{eq:odeg}
\wedge^4_+ V^{\ast} = \wedge^4_1 V^{\ast} \oplus \wedge^4_7 V^{\ast} \oplus \wedge^4_{27} V^{\ast} 
\een
The subscript in this decomposition denotes the real dimension of the corresponding irreducible representation of $\Spin(7)$. Notice that $\Phi\in \wedge^4_1 V^\ast$.  Since this subspace is one-dimensional, we have: 
\be
\wedge^4_1 V^\ast=\R \Phi 
\ee
Thus any $q\in \wedge^4_1 V^\ast$ is of the form $q=\lambda(q) \Phi$ with $\lambda(q)\in \R$. Taking norms gives $|\lambda(q)|=\frac{|q|_h}{|\Phi|_h}$ and hence:
\ben
\label{q1}
q=\epsilon(q) \frac{|q|_h}{|\Phi|_h}\Phi \quad \forall q\in \wedge^4_1 V^\ast~,
\een
where $\epsilon(q):=\sign(\lambda(q))$. As Shown in \cite[Proposition 2.8]{KarigiannisFlows} (see also \cite[Proposition 2.6]{DwivediGrad}), the invariant subspaces appearing in the right-hand side of \eqref{eq:odeg} can be written as eigenspaces of the operator $\Lambda^h_\Phi$ of  Definition \ref{def:Lambda} (see Remark \ref{rem:Lambda}):
\beqa
& \wedge^4_1 V^\ast = \left\{q\in \wedge^4_{+} V^{\ast}\,\, \vert\,\, \Phi \Delta^{h_{\Phi}}_2 q = -12 q \right\}\, , \quad 
\wedge^4_7 V^\ast = \left\{ q \in \wedge^4_{+} V^{\ast}\,\, \vert\,\, \Phi \Delta^{h_{\Phi}}_2 q = - 6 q \right\}\\
& \wedge^4_{27} V^\ast = \left\{ q\in \wedge^4_{+} V^{\ast}\,\, \vert\,\, \Phi \Delta^{h_{\Phi}}_2 q= 2 q\right\} 
\eeqa
We also have $\wedge^4_{-} V^\ast=\wedge^4_{35} V^\ast=\ker \Lambda_\Phi^h$. Using Lemma \ref{lemma:conformalconstant}, we write the relations above in terms of $h$:
\beqan
\label{InvSubspaces}
& \wedge^4_1 V^\ast = \left\{ q \in \wedge^4_{+} V^{\ast}\,\, \vert\,\, \Phi\Delta^{h}_2 q = -  \frac{12}{\sqrt{14}} \,\vert \Phi\vert_h q \right\}\, , \quad \wedge^4_7 V^\ast = \left\{ q \in \wedge^4_{+} V^{\ast}\,\, \vert\,\, \Phi\Delta^{h}_2 q = -  \frac{6}{\sqrt{14}} \,\vert \Phi\vert_h  q \right\}\nn\\
& \wedge^4_{27} V^\ast = \left\{q\in \wedge^4_{+} V^{\ast}\,\, \vert\,\, \Phi \Delta^{h}_2 q=  \frac{2}{\sqrt{14}} \,\vert \Phi\vert_h q \right\} 
\eeqan

\begin{prop}
\label{prop:ddWh}
Let $\Phi\in \wedge^4_+ V^\ast$ be a conformal $\Spin(7)_+$ form on $(V,h)$. Then $\dd^2_\Phi W_h$ is block-diagonal with respect to the decomposition \eqref{eq:odeg}. Moreover, the restrictions of $\dd^2_\Phi W_h$ to the subspaces appearing in this decomposition are given by: 
\ben
\label{HessRestrictions}
(\dd^2_{\Phi}W_h)(q_1,q_2) =
\threepartdef{0}{q_1,q_2\in \wedge^4_1 V^\ast}{0}{q_1,q_2\in \wedge^4_7 V^\ast}{16 |\Phi|_h \langle q_1,q_2\rangle_h}{q_1 , q_2\in \wedge^4_{27} V^\ast} 
\een
\end{prop}

\begin{proof}
Since $W_h$ is $\SO(V,h)$-invariant, its second differential $\dd^2_\Phi W_h:\wedge^4_+ V^\ast\otimes \wedge^4_+V^\ast\rightarrow \R$ computed at $\Phi$ is invariant under the tensor product action of the $\Spin(7)_+$ stabilizer of $\Phi$ on the argument $q_1\otimes q_2\in \wedge^4_+ V^\ast \otimes \wedge^4_+ V^\ast$. By Schur's lemma, this implies that $\dd^2_{\Phi}W_h$ is block-diagonal relative to the decomposition \eqref{eq:odeg}. To compute the restrictions to each subspace, it suffices to compute the restrictions of the diagonal of the symmetric bilinear form $\dd^2_{\Phi}W_h$, which is given by (see \eqref{ddWh}): 
\ben
\label{HessDiag}
(\dd^2_{\Phi}W_h)(q,q)  =  2\sqrt{14}\,\langle  q\Delta_2^h q  , \Phi \rangle_h + 12 \vert \Phi\vert_h  |q|^2_h + \frac{12}{|\Phi|_h}  |\langle q , \Phi\rangle_h|^2 \quad \forall q\in \wedge^4_+ V^\ast 
\een
Using Lemma \ref{lemma:cyc4form} and relations \eqref{InvSubspaces}, we find: 
\be
\sqrt{14}\langle  q\Delta_2^h q  , \Phi \rangle_h=\sqrt{14}\langle  \Phi \Delta_2^h q  , q \rangle_h=
\threepartdef{-12 |\Phi|_h |q|_h^2}{q\in \wedge^4_1 V^\ast}{-6 |\Phi|_h |q|_h^2}{q\in \wedge^4_7 V^\ast}{+2 |\Phi|_h |q|_h^2}{q\in \wedge^4_{27} V^\ast} 
\ee
Substituting this in \eqref{HessDiag} gives: 
\ben
\label{HessDiagRestructions|}
(\dd^2_{\Phi}W_h)(q,q) =
\threepartdef{0}{q\in \wedge^4_1 V^\ast}{0}{q\in \wedge^4_7 V^\ast}{16|\Phi|_h |q|_h^2}{q\in \wedge^4_{27} V^\ast}~,
\een
where for $q\in \wedge^4_1 V^\ast$ we used \eqref{q1}. The relation \eqref{HessRestrictions} follows by polarizing the identities above. 
\end{proof}

\noindent
Proposition \ref{prop:ddWh} shows that the Hessian of $W_h$ at a conformal $\Spin(7)_+$ form $\Phi \in \cC_+(V,h)$ restricts to zero on the tangent space at $\Phi$ to the cone $C_+(V,h)$, which decomposes as: 
\be
T_\Phi C_+(V,h)=\wedge^4_1 V^\ast \oplus \wedge^4_7 V^\ast
\ee
and that it is positive definite when restricted to the normal space to the cone at the point $\Phi$, which is given by: 
\be
N_\Phi C_+(V,h)=[T_\Phi C_+(V,h)]^\perp=\wedge^4_{27} V^\ast 
\ee
The fact that the Hessian must restrict to zero on $T_\Phi C_+(V,h)$ is clear from the fact that $W_h$ is homogeneous under positive homotheties of $\Phi$ and invariant under the action $\SO(V,h)$ on $\wedge^4_+ V^\ast$. Since the critical locus of $W_h$ coincides with $C_+(V,h)$, we obtain the following result.

\begin{cor}
\label{cor:Hessian}
The potential $W_h$ is a Morse-Bott function on the complement of the origin in $\wedge^4_+ V^\ast$. Moreover, the restriction of the differential $\dd^2 W_h$ to the normal bundle of any generator of the cone $C_+(V,h)$ is positive-definite. In particular, the critical set of $W_h$ consists of local minima of the potential. The critical set is contained in the zero locus of $W_h$. It coincides with the cone $C_+(V,h)$ and consists of the origin and the conformal $\Spin(7)_+$ forms of $(V,h)$. 
\end{cor}

\noindent Note that $W_h$ is not coercive on $\wedge^4_+ V^\ast$, since its restriction to the cone $C_+(V,h)$ is identically zero. Restricting to the sphere $S_{\sqrt{14}}(\wedge^4_+ V^\ast)$, we find: 

\begin{cor}
\label{cor:HessianMetric}
The restriction:
\begin{equation*}
{\hat W}_h\colon S_{\sqrt{14}}(\wedge^4_+ V^\ast)\rightarrow \R
\end{equation*}
of $W_h$ to the 34-dimensional sphere $S_{\sqrt{14}}(\wedge^4_+ V^\ast)$ of radius $\sqrt{14}$ in the 35-dimensional Euclidean space $(\wedge^4_+ V^\ast, \langle~,~\rangle_h)$ has a local minimum at each metric $\Spin(7)_+$ form on $(V,h)$ and the restricted potential ${\hat W}_h$ vanishes at each such point. 
\end{cor}

\noindent
We end this section by pointing out a bound satisfied by $W_h$. Write the potential as: 
\ben
\label{Wf}
W_h(\Phi)=\sqrt{14}f_h(\Phi)+4|\Phi|_h^3~,
\een
where $f_h:\wedge_+^4 V^\ast\rightarrow \R$ is the cubic function defined through: 
\be
f_h(\Phi):=-\frac{1}{3}\langle \Phi, \Phi\diamond_h\Phi\rangle_h=\frac{1}{3} \langle \Phi, \Phi\Delta_2^h \Phi\rangle_h \quad \forall \Phi\in \wedge^4_+ V^\ast 
\ee
Notice that: 
\ben
\label{antisymm}
f(-\Phi)=-f(\Phi)\quad \forall \Phi\in \wedge^4_+ V^\ast 
\een

\begin{lemma}
\label{lemma:f}
For any nonzero self-dual four-form $\Phi$, we have: 
\be
|f_h(\Phi)|\leq \frac{1}{3}|\Phi|_h |\Phi\Delta_2^h \Phi|_h<\frac{\sqrt{14}}{3} |\Phi|_h^3~,
\ee
where the first inequality is saturated if and only if $\Phi$ satisfies the equation:
\ben
\label{Phieq}
\Phi\Delta_2^h \Phi = c \Phi 
\een
for some $c\in \R$. In this case, we have: 
\ben
\label{cval}
c=\frac{3f(\Phi)}{|\Phi|_h^2} 
\een
\end{lemma}

\begin{proof}
By the Cauchy-Schwarz inequality and Proposition \ref{prop:bound}, we have:  
\be
|f_h(\Phi)|=\frac{1}{3}|\langle \Phi, \Phi\Delta_2^h \Phi\rangle_h|\leq \frac{1}{3}|\Phi|_h |\Phi\Delta_2^h \Phi|_h<\frac{\sqrt{14}}{3} |\Phi|_h^3~,
\ee
where the first inequality is saturated when: 
\ben
\label{Phieq0}
\Phi\Delta_2^h \Phi =c \Phi
\een
for some $c\in \R$. Taking the scalar product with $\Phi$ in both sides of this equation gives \eqref{cval}.
\end{proof}

\begin{prop}
\label{prop:Whbound}
For any nonzero self-dual four-form $\Phi$, we have: 
\be
W_h(\Phi)\leq \frac{\sqrt{14}}{3}|\Phi|_h |\Phi\Delta_2^h \Phi|_h+4 |\Phi|_h^3<\frac{26}{3} |\Phi|_h^3~,
\ee
where the first inequality is saturated if and only if $\Phi$ satisfies the conditions of Lemma \ref{lemma:f}.
\end{prop}


\subsection{Metric deformations of the potential}
\label{sec:algebraicmetricdef}


In this section, we study the expression in the right hand side of \eqref{eq:cubicfunction} as a function of pairs $(h,\Phi)$, where $h$ is an Euclidean metric on $V$ and $\Phi\in \wedge^4 V^{\ast}$ is a four-form that need not be self-dual with respect to $h$. We thus consider the function:
\ben
\label{eq:metriccubicfunction}
W\colon \Met(V)\times\wedge^4 V^{\ast} \rightarrow \mathbb{R}\, ,\quad (h,\Phi) \mapsto  W(h,\Phi) :=\frac{\sqrt{14}}{3}\langle \Phi\Delta_2^h\Phi , \Phi \rangle_h + 4 \langle \Phi , \Phi \rangle_h^{\frac{3}{2}}~, 
\een

\noindent
where $\Met(V)$ denotes the cone of Euclidean metrics on $V$. 

\begin{definition}
A pair $(h,\Phi)\in \Met(V)\times\wedge^4 V^{\ast}$ is called \emph{self-dual} if $\Phi$ is self-dual relative to $h$. A pair $(h,\Phi)\in \Met(V)\times\wedge^4 V^{\ast}$ is called a \emph{conformal $\Spin(7)_+$ pair} on the oriented eight-dimensional space $V$ if $\Phi$ is  conformal $\Spin(7)_+$ form on the oriented Euclidean vector space $(V,h)$.
\end{definition}

\noindent
We denote by:
\begin{equation*}
\dd_{(h,\Phi)} W\colon V^{\ast}\odot V^{\ast} \oplus \wedge^4 V^{\ast} \to \mathbb{R}
\end{equation*}

\noindent
the differential of $W$ at the point $(h,\Phi)\in \Met(V)\times \wedge^4 V^{\ast}$. 

\begin{lemma}
\label{lemma:metricdeformations}
Let $(h,\Phi)\in \Met(V)\times \wedge^4 V^{\ast}$ be a conformal Spin(7) structure on $V$. Then:
\begin{equation*}
(\dd_{(h,\Phi)} W)(k,0) = 0
\end{equation*}

\noindent
for every $k\in V^{\ast}\odot V^{\ast}$.
\end{lemma}

\begin{proof}
First, recall that by Equation \eqref{PhiDeltaomega} we have:
\be
\Phi\Delta_2^h \Phi=\frac{1}{8}\Phi_{ijmn}\Phi^{mn}_{\,\,\,\,\,\,\,\,\,\, kl} e^i\wedge e^j\wedge e^k\wedge e^l
\ee

\noindent
for every $\Phi\in \wedge^4 V^\ast$ and any basis $(e^1,\hdots ,e^8)$. Hence:
\begin{equation*}
\langle \Phi\Delta_2^h\Phi , \Phi \rangle_h = \frac{1}{8}\Phi_{ijmn}\Phi^{mn}_{\,\,\,\,\,\,\,\,\,\, kl} \Phi^{ijkl} = \frac{1}{8} \Phi_{i_1 i_2 i_3 i_4}\Phi_{j_1 j_2 j_3 j_4}\Phi_{k_1 k_2 k_3 k_4} h^{i_1 k_1}h^{i_2 k_2}h^{i_3 j_1}h^{i_4 j_2}h^{j_3 k_3}h^{j_4 k_4}
\end{equation*}

\noindent
as well as:
\be
\vert\Phi\vert^2_h = \frac{1}{4!} \Phi_{ijkl} \Phi^{ijkl} = \frac{1}{4!} \Phi_{i_1 i_2 i_3 i_4} \Phi_{j_1 j_2 j_3 j_4} h^{i_1 j_1}h^{i_2 j_2}h^{i_3 j_3}h^{i_4 j_4}
\ee

\noindent
Therefore:
\be
W(h,\Phi) = \frac{\sqrt{14}}{24} \Phi_{i_1 i_2 i_3 i_4}\Phi_{j_1 j_2 j_3 j_4}\Phi_{k_1 k_2 k_3 k_4} h^{i_1 k_1}h^{i_2 k_2}h^{i_3 j_1}h^{i_4 j_2}h^{j_3 k_3}h^{j_4 k_4} + \frac{4}{(4!)^{\frac{3}{2}}} (\Phi_{i_1 i_2 i_3 i_4} \Phi_{j_1 j_2 j_3 j_4} h^{i_1 j_1}h^{i_2 j_2}h^{i_3 j_3}h^{i_4 j_4})^{\frac{3}{2}} 
\ee 

\noindent
Using this expression for $W(h,\Phi)$, which we have written explicitly in terms of all the metric contractions for clarity in the exposition, a direct computation gives:
\beqa
& (\dd_{(h,\Phi)} W)(k,0) = - \frac{\sqrt{14}}{4} \Phi_{i_1 i_2 i_3 i_4}\Phi_{j_1 j_2 j_3 j_4}\Phi_{k_1 k_2 k_3 k_4} k^{i_1 k_1}h^{i_2 k_2}h^{i_3 j_1}h^{i_4 j_2}h^{j_3 k_3}h^{j_4 k_4} \\
& - \vert\Phi\vert_h  \Phi_{i_1 i_2 i_3 i_4} \Phi_{j_1 j_2 j_3 j_4} k^{i_1 j_1}h^{i_2 j_2}h^{i_3 j_3}h^{i_4 j_4} = - 4^{-1} 14^{-\frac{3}{4}} \vert\Phi\vert_h^{\frac{5}{2}}\Phi_{i_1 i_2 i_3 i_4}\Phi_{j_1 j_2 j_3 j_4}\Phi_{k_1 k_2 k_3 k_4} k^{i_1 k_1}h^{i_2 k_2}_{\Phi} h^{i_3 j_1}_{\Phi} h^{i_4 j_2}_{\Phi} h^{j_3 k_3}_{\Phi} h^{j_4 k_4}_{\Phi}\\
& - 14^{-\frac{3}{4}} \vert\Phi\vert^{\frac{5}{2}}_h  \Phi_{i_1 i_2 i_3 i_4} \Phi_{j_1 j_2 j_3 j_4} k^{i_1 j_1}h^{i_2 j_2}_{\Phi} h^{i_3 j_3}_{\Phi} h^{i_4 j_4}_{\Phi}
\eeqa 

\noindent
where we have used Equation \eqref{hPhi} to relate $h$ and the metric $h_{\Phi}$ induced by $\Phi$. Since $h_{\Phi}$ is induced by $\Phi$ and the latter is a Spin(7) structure, the following well-known identities hold \cite{KarigiannisDefs,KarigiannisFlows}:
\beqa
& \Phi_{i_1 i_2 i_3 i_4} \Phi_{j_1 j_2 j_3 j_4} h^{i_3 j_3}_{\Phi} h^{i_4 j_4}_{\Phi} = 6 (h_{\Phi})_{i_1 j_1} (h_{\Phi})_{i_2 j_2} - 6 (h_{\Phi})_{i_1 j_2} (h_{\Phi})_{i_2 j_1} - 4 \Phi_{i_1 i_2 j_1 j_2}\\
& \Phi_{i_1 i_2 i_3 i_4} \Phi_{j_1 j_2 j_3 j_4} h^{i_2 j_2}_{\Phi} h^{i_3 j_3}_{\Phi} h^{i_4 j_4}_{\Phi} = 42 (h_{\Phi})_{i_1 j_1}
\eeqa

\noindent
Using these identities, we obtain:
\beqa
& (\dd_{(h,\Phi)} W)(k,0)  =  14^{-\frac{3}{4}} 42 \vert\Phi\vert_h^{\frac{5}{2}}  k^{i_1 k_1} (h_{\Phi})_{i_1 j_1} - 14^{-\frac{3}{4}} 42 \vert\Phi\vert^{\frac{5}{2}}_h  k^{i_1 k_1} (h_{\Phi})_{i_1 j_1} = 0
\eeqa 

\noindent
and thus we conclude.
\end{proof}

\noindent
Using this lemma we adapt Theorem \ref{thm:PotentialVh} to the case of the function $W\colon \Met(V)\times\wedge^4 V^{\ast} \rightarrow \mathbb{R}$.

\begin{thm}
\label{thm:MetricPotentialVh}
A pair $(h,\Phi)\in \Met(V)\times \wedge^4 V^{\ast}$ is a conformal Spin(7) structure on $V$ if and only if it is a self-dual critical point of the function $W\colon \Met(V)\times\wedge^4 V^{\ast} \rightarrow \mathbb{R}$. 
\end{thm}

\begin{proof}
Let $(h,\Phi)\in \Met(V)\times \wedge^4 V^{\ast}$ be a self-dual pair. Then, by Corollary \ref{cor:dWh} we have: 
\be
(\dd_{(h,\Phi)} W) (0,q)=\langle\sqrt{14}\Phi\Delta_2^h\Phi + 12 |\Phi|_h \Phi, q\rangle_h \quad \forall q\in \wedge^4_+ V^\ast
\ee

\noindent
Hence, condition $(\dd_{(h,\Phi)} W) (0,q) = 0$ holds for all $q\in \wedge^4_+ V^\ast$ if and only if $\langle\sqrt{14}\Phi\Delta_2^h\Phi + 12 |\Phi|_h \Phi, q\rangle_h = 0$, which by Theorem \ref{thm:Spin7algebraic} is equivalent to $\Phi$ being a conformal Spin(7) structure on $(V,h)$. Since $(g,\Phi)$ is a conformal Spin(7) structure, Lemma \ref{lemma:metricdeformations} implies that $(\dd_{(h,\Phi)} W) (k,0) = 0$ for every $k\in V^{\ast}\odot V^{\ast}$ and thus we conclude.
\end{proof}

\noindent
Hence by considering the variations of $W\colon \Met(V)\times \wedge^4 V^{\ast} \to \mathbb{R}$ with respect to pairs $(h,\Phi)$ we can describe all 
$\Spin(7)_+$ structures on $V$, as opposed to those that are conformal relative to a fixed Euclidean metric. It would be interesting to investigate the geometric significance of those critical points of $W$ which are not self-dual.


\section{Conformal Spin(7) forms on compact eight-manifolds}


In this section, we globalize the algebraic results of the previous sections to an oriented smooth eight-manifold $M$. Given a Riemannian metric $g$ on $M$, we will denote its associated Riemannian volume form by $\nu_g$. Theorem \ref{thm:Spin7algebraic} immediately implies the following result on $(M,g)$.

\begin{cor}
A nowhere-vanishing smooth self-dual four-form $\Phi\in \Omega^4_+(M)$ is a conformal $\Spin(7)_+$ form on $(M,g)$ if and only if it satisfies the degree-two homogeneous equation:
\be
\sqrt{14} \Phi\Delta_2^g \Phi + 12\vert\Phi\vert_g \Phi = 0 
\ee
In this case, the Riemannian metric induced by $\Phi$ is $g_\Phi=14^{-\frac{1}{4}} \vert \Phi\vert_g^{\frac{1}{2}}g$.
\end{cor}

\noindent 
Let $\wedge^4_+ T^\ast M$ be the vector bundle of self-dual four-forms on $(M,g)$, whose space of sections we denote by $\Omega^4_+(M)$. The potential $W_h$ of Section \ref{sec:Wh} globalizes to a function $W_g\colon \wedge^4_+ T^\ast M \rightarrow \R$ which vanishes on the zero section and restricts on each fiber to a function which is homogeneous of degree three. The vertical critical locus of $W_g$ defines a bundle of cones $C_+(M,g)\subset \wedge^4_+ T^{\ast}M$ that is fiber-wise isomorphic to $C_+(V,h)$. The conformal $\Spin(7)_+$ forms on $(M,g)$ are sections of the bundle of pointed cones $C_+(M,g)\backslash \left\{ 0\right\}$ obtained from $C_+(M,g)$ by removing the zero section of $\wedge^4_+ T^\ast M$.

For the remainder of this section, we assume that $M$ is compact and without boundary. Consider the following functional defined on the Fréchet space $\Omega^4_{+}(M)$ of smooth self-dual four-forms over $(M,g)$, which is the natural global version of the cubic potential \eqref{eq:cubicfunction}:
\be
\cW_g \colon  \Omega^4_{+}(M) \to \mathbb{R}\, ,\quad \Phi \to \cW_g(\Phi) := \int_M  \left[\frac{\sqrt{14}}{3}\langle \Phi\Delta_2^g \Phi , \Phi \rangle_g + 4 \langle \Phi , \Phi \rangle^{\frac{3}{2}}_g\right] \nu_g~. 
\ee
Here $\langle~,~\rangle_g$ is the fiberwise Euclidean pairing induced by $g$ on $\wedge T^\ast M$. Since the restriction of the integrand of $W_g$ to the fibers of the vector bundle $\wedge_+ T^\ast M$ of self-dual four-forms coincides with the cubic potential of the previous section, Theorem \ref{thm:PotentialVh} implies the following:

\begin{cor}
A smooth nowhere-vanishing self-dual four-form is a smooth conformal $\Spin(7)_+$ form on $(M,g)$ if and only if it is a critical point of 
the functional $\cW_g \colon  \Omega^4_{+}(M) \to \mathbb{R}$. Moreover, $\cW_g$ vanishes when evaluated on any smooth conformal $\Spin(7)_+$ form. 
\end{cor}

\noindent
As in Subsection \ref{sec:algebraicmetricdef}, we extend $\cW_g$ to a functional $\cW \colon \Met(M)\times \Omega^4(M) \to \mathbb{R}$ defined on arbitrary pairs $(g,\Phi)\in \Met(M)\times \Omega^4(M)$ by the same formula:
\ben
\label{eq:potentialmetricspin7}
\cW(g,\Phi) := \int_M  \left[\frac{\sqrt{14}}{3}\langle \Phi\Delta_2^g \Phi , \Phi \rangle_g + 4 \langle \Phi , \Phi \rangle^{\frac{3}{2}}_g\right] \nu_g~, 
\een

\noindent
where $\Met(M)$ denotes the convex cone of Riemannian metrics on $M$. 
\begin{definition}
A pair $(g,\Phi)\in \Met(M)\times\Omega^4(M)$ is called \emph{self-dual} if the four-form $\Phi$ is self-dual with respect to the metric $h$. A pair $(g,\Phi)\in \Met(M)\times\Omega^4(M)$ is called a \emph{conformal $\Spin(7)_+$ pair} on $M$ if $\Phi$ is a conformal $\Spin(7)_+$ form on the Riemannian manifold $(M,g)$.
\end{definition}

\noindent
Theorem \ref{thm:MetricPotentialVh} implies the following:

\begin{cor}
\label{cor:MetricPotentialVh}
Suppose that the eight-manifold $M$ is compact and spin. Then a self-dual pair $(g,\Phi)\in \Met(M)\times \Omega^4(M)$ is a conformal $\Spin(7)_+$ pair on $M$ if and only if it is a critical point of the functional $\cW\colon \Met(M)\times\Omega^4(M) \rightarrow \mathbb{R}$. 
\end{cor}

\noindent
Corollary \ref{cor:MetricPotentialVh} gives a variational description of all $\Spin(7)_+$ structures on a compact spin manifold, modulo the self-duality condition that needs to be checked independently. In particular, every $\Spin(7)_+$ form $\Phi$ identifies canonically with a point of the set $\Crit_{+}(\cW)$ of self-dual pairs which are critical points of $\cW$ as follows:
\begin{equation*}
\Phi \mapsto (g_{\Phi}, \Phi) \in \Crit_{+}(\cW)
\end{equation*}

\noindent
where $g_{\Phi}$ is the Riemannian defined by $\Phi$ on $M$. It would be interesting to study the geometric significance of the non-self-dual critical points of $\cW$, which might provide a weakening of the notion of $\Spin(7)_+$ structure on $M$.

\begin{remark}
Of course, there is no assurance in general that any smooth conformal $\Spin(7)_+$ structure exists on $(M,g)$ --- and hence the critical locus of $\cW_g$ or $\cW$ may be empty. In fact, it is well-known that the following statements are equivalent for a compact and spin Riemannian eight-manifold  $(M,g)$ whose half spinor bundles we denote by $S^+$ and $S^-$ (see \cite[Theorem 10.7]{LM}):
\begin{itemize}
\item $M$ admits a $\Spin(7)$ form.
\item $M$ is spin and either $\chi(S^{+}) = 0$ or $\chi(S^{-}) = 0$.
\item The following relation holds for an appropriate choice of orientation on $M$:
\begin{equation}
\label{eq:topologicalcondition}
p_1(M)^2 - 4 p_2 (M) + 8 \chi(M) = 0~,
\end{equation}
where $\chi(M)$ is the Euler class of $M$, $\chi(S^{\pm})$ is the Euler class of $S^{\pm}$, and $p_1(M)$ is the first Pontryagin class of $M$. 
\end{itemize}
It follows that the existence of critical points of the functional $\cW_g$ are obstructed by the topological condition \eqref{eq:topologicalcondition}. 
\end{remark}

We note that the potential $\cW_g$ extends naturally to a function defined on the Banach space $\L^3(\wedge^4_+ T^\ast M)$ of self-dual cubic integrable four-forms on $(M,g)$. This follows by writing $\cW_g$ as: 
\ben
\label{cWdiamond}
\cW_g(\Phi)= \int_M  \left[\frac{\sqrt{14}}{3}\langle \Phi\Delta_2^h\Phi, \Phi \rangle_g + 4 |\Phi|^3_g\right] \nu_g
\een
and noticing that the Cauchy-Schwarz inequality and Corollary \ref{cor:normdiamond} imply: 
\be
|\langle \Phi\Delta_2^h\Phi, \Phi\rangle_g|_g\leq \sqrt{16}|\Phi|_g^3 
\ee
Using the generalized H\"{o}lder inequality, it follows that the integrands of the two terms appearing in \eqref{cWdiamond} are absolutely integrable functions (elements of $\L^1(M,g)$) provided that $\Phi\in \L^3(\wedge^4_+ T^\ast M)$. 

We end this section by considering a natural coupling of the functional \eqref{eq:potentialmetricspin7} to the Einstein-Hilbert action. Consider the following functional which depends on an arbitrary constant $\lambda\in \mathbb{R}$:
\ben
\label{eq:Einsteinpotential}
\cL_{\lambda}(g,\Phi) := \int_M  \left[ s^g + \frac{\sqrt{14}}{3} \lambda \langle \Phi\Delta_2^h\Phi , \Phi \rangle_g + 4 \lambda |\Phi|^3_g\right] \nu_g~.
\een

\noindent
Here $s^g$ is the scalar curvature of the Riemannian metric $g\in \Met(M)$ and $\Phi\in \Omega^4(M)$ is any four-form. Lemma \ref{lemma:metricdeformations}, implies the following characterization of the self-dual critical points of $\cL_{\lambda}$.

\begin{prop}
\label{prop:EinsteinSpin7}
A self-dual pair $(g,\Phi)\in \Met(M)\times \Omega^4(M)$ is a critical point of the functional:
\be
\cL_{\lambda} \colon \Met(M)\times \Omega^4(M) \to \mathbb{R}
\ee

\noindent
if and only if the metric $g$ is Ricci flat and $\Phi$ is a conformal $\Spin(7)_+$ structure on $(M,g)$.
\end{prop}

\noindent
Exploring this potential in more detail could be interesting, especially in the context of gradient flows of Spin(7) structures. The geometric or topological significance of the non-self-dual critical points of $\cL_{\lambda}$, if any, remains to be ascertained.


\appendix


\end{document}